\documentclass[12pt]{amsart}

\usepackage{color}
\usepackage{amsfonts}
\usepackage{amscd}
\usepackage{amssymb}
\usepackage{amsthm}
\usepackage{amsmath}
\usepackage{comment}
\usepackage{epsfig}
\usepackage[all]{xy}

\usepackage{mathrsfs}

\usepackage[colorlinks=true, pdfstartview=FitV, linkcolor=blue, citecolor=blue, urlcolor=blue]{hyperref}
\usepackage{tikz,etex,bbm,xspace,hyperref}
\newcommand{\arxiv}[1]{\href{http://arxiv.org/abs/#1}{\tt arXiv:\nolinkurl{#1}}}

\usepackage{amssymb}

\newtheorem{Theorem}{Theorem}[section]
\newtheorem{Proposition}[Theorem]{Proposition}
\newtheorem{Lemma}[Theorem]{Lemma}
\newtheorem{Nested Lemma}[Theorem]{Nested Lemma}

\newtheorem{Corollary}[Theorem]{Corollary}

\theoremstyle{definition}

\newtheorem{Definition}[Theorem]{Definition}

\theoremstyle{remark}
\newtheorem{Remark}[Theorem]{Remark}

\newcommand{\pcom}{\color{blue}}

\newcommand{\preci}{{\prec^{s_i}}}

\DeclareMathOperator{\sgn}{sgn}

\makeatletter
\renewcommand{\@makefnmark}{\mbox{\textsuperscript{}}}
\makeatother

\newcommand{\B}{\mathfrak{B}}

\newcommand{\cL}{\mathcal{L}}
\newcommand{\coa}{\mathbbm{c}}

\newcommand{\wt}{\text{wt}}

\newcommand{\eps}{\varepsilon}

\newcommand{\cc}{\mathbf{c}}
\newcommand{\dd}{\mathbf{d}}

\newcommand{\cB}{\mathcal{B}}
\renewcommand{\cL}{\mathcal{L}}

\newcommand{\e}{\tilde{e}}
\newcommand{\f}{\tilde{f}}
\newcommand{\tpsi}{\tilde{\psi}}
\newcommand{\cA}{\mathcal{A}}
\newcommand{\ATT}{A_2^{(2)}}
\newcommand{\fg}{\mathfrak{g}}
\newcommand{\fh}{\mathfrak{h}}

\newcommand{\nc}{\newcommand}
\nc{\cosoc}{\operatorname{cosoc}}
\nc{\soc}{\operatorname{soc}}
\nc{\asl}{\widehat{\mathfrak{sl}}}
\nc{\g}{\mathfrak{g}}
\nc{\br}{\Bbb R}
\nc{\bz}{\Bbb Z}
\nc{\bn}{\Bbb N}
\nc{\omu}{\overline{\mu}}
\nc{\olambda}{\overline{\lambda}}
\nc{\oa}{\overline{a}}
\nc{\Irr}{\text{Irr }}
\nc{\oT}{\overline{T}}
\nc{\oR}{\overline{R}}
\nc{\oI}{\overline{I}}
\nc{\bfv}{{\bf v}}
\nc{\vareps}{\varepsilon}
\nc{\MV}{\mathcal{MV}}
\nc{\be}{\beta}
\nc{\End}{\text{End}}

\newcommand{\PBW}{{\text{PBW}}}

\newcommand{\U}{U_q(\g)}
\newcommand{\Uplus}{U^+_q(\g)}
\newcommand{\QQ}{\mathbb{Q}}
\newcommand{\ZZ}{\mathbb{Z}}
\newcommand{\Uplusintegral}{{}_{\mathscr{A}}\Uplus}
\newcommand{\IPsi}[2]{\Psi^{#1}_{#2}}

\begin{document}

\title{Affine PBW Bases and Affine MV polytopes}

\author{Dinakar Muthiah}
\address{Dinakar Muthiah, Dept. of Mathematical and Statistical Sciences, University of Alberta}
\email{muthiah@ualberta.ca}

\author{Peter Tingley}
\address{Peter Tingley, Dept. of Maths and Stats, Loyola University Chicago}
\email{ptingley@luc.edu}

\begin{abstract}

We show how affine PBW bases can be used to construct affine MV polytopes, and that the resulting objects agree with the affine MV polytopes recently constructed using either preprojective algebras or KLR algebras. To do this we first generalize work of Beck-Chari-Pressley and Beck-Nakajima to define affine PBW bases for arbitrary convex orders on positive roots. Our results describe how affine PBW bases for different convex orders are related, answering a question posed by Beck and Nakajima. 

\end{abstract}

\maketitle

\tableofcontents

\section{Introduction}

Mirkovi\'c-Vilonen (MV) polytopes were developed by Kamnitzer \cite{Kam1,Kam2} and Anderson \cite{And} to study complex simple Lie algebras and their finite-dimensional representations. 
They arose from the geometry of the affine Grassmannian as developed in \cite{MirkovicVilonen04}, but are now known to appear in a number of other places, such as in the representation theory of preprojective algebras \cite{BK,BKT} and of KLR algebras \cite{TW}. 
Much of the combinatorics was also developed by Lusztig in the early 1990s while studying PBW bases, and here we are interested in that point of view.

Lusztig \cite{LusCanonical} defined a family of PBW bases for $U_q^+(\g)$, the positive part of the quantized universal enveloping algebra associated to the Lie algebra $\g$, which depend on a chosen reduced expression of the longest element $w_0$ of the Weyl group. PBW bases are crystal bases in the sense of Kashiwara~\cite{Kashiwara:1995}, so each basis is in canonical bijection with the crystal $B(-\infty)$. As shown in \cite{Kam1}, these bijections are encoded by MV polytopes as follows: For each $b \in B(-\infty)$, consider the corresponding MV polytope $MV_b$. Each reduced decomposition of $w_0$ corresponds to a path through $MV_b$, and the lengths of the edges in that path determine the PBW basis vector corresponding to $b$. 

There is now a notion of MV polytopes for affine Kac-Moody algebras \cite{BDKT:13, BKT, TW}. 
PBW bases for affine Kac-Moody algebras have also been constructed (see works of Beck-Chari-Pressley \cite{BCP} and Beck-Nakajima \cite{BN}). In \cite{MT} we established that, in rank two, affine MV polytopes and affine PBW bases are related as one would expect. Here we extend that result to include all affine types. 
In fact, this gives a way to define affine MV polytopes algebraically, without having to introduce the auxiliary objects (quiver varieties or KLR algebras) needed in previous approaches. 

In affine type there is no longest element, so we cannot discuss reduced expressions of $w_0$. Instead we consider convex orders on the set of positive roots. 
One difficulty is that the construction of PBW bases in \cite{BCP,BN} is only given  for certain special convex orders. We first extend their definitions to include all convex orders, and show that the construction is independent of certain choices (this has also been done very recently by McNamara \cite{McN} using KLR algebras). 
We use PBW bases to define a map from $B(-\infty)$ to decorated polytopes where, as in finite type, the exponents of the PBW monomials are encoded as the edge lengths (and decoration) along certain paths through the polytopes (see Theorem \ref{thm:ispw}). We then show that our PBW polytopes agree with the MV polytopes from \cite{BKT, TW} (see Theorem \ref{thm:PBWMV}).
We do this by giving a characterization of the map taking $b \in B(-\infty)$ to its PBW polytope (Theorem \ref{th:MV-def}) which we think is interesting in its own right.

\subsection{Acknowledgements}
We thank Joel Kamnitzer and Peter McNamara for many useful discussions. D.M. was supported by a PIMS Postdoctoral Fellowship. P.T. was partially supported by NSF grant DMS-1265555. 

\section{Background}

\subsection{Notation and conventions}

\begin{itemize}

\item $I=\{ 0, \ldots, n\}$ is the index set of a connected affine Dynkin diagram, where $0$ is the distinguished node as in \cite[Section 2.1]{BN}. Let $\bar I = I \backslash \{0\}$; this is a finite-type Dynkin diagram. These conventions agrees with the notation in \cite[Section 4.8]{Kac} except in the case of a Dynkin diagram of type $A^{(2)}_{2n}$. Let $(a_{i,j})_{i,j\in I}$ denote the corresponding Cartan matrix.

\item $\g$ is the corresponding affine Kac-Moody algebra, and $\bar \g$ be the finite-type Lie algebra corresponding to $\bar I$. Fix a triangular decomposition $\g = \mathfrak{n_-} \oplus \mathfrak{h} \oplus \mathfrak{n_+}$ (coming from the construction of $\g$ by generators and relations). 

\item $W$ is the Weyl group for $\g$ and $\overline{W}$ is the Weyl group for $\bar \g$.

\item $\Delta$ is the affine root system of $\g$ and $\bar \Delta$ the root system of $\bar \g$. We denote the simple roots by $\alpha_i$ and $\bar \alpha_i$ respectively.

\item $\Delta^+$ is the set of positive roots of $\g$.

\item $\Delta_+^{min}$ is the set of positive roots $\alpha$ such that $x \alpha$ is not a root for any $0 < x <1$. Since we restrict to affine type, $\Delta_+^{min}$ consists of the positive real roots along with the minimal imaginary root $\delta$. 

\item $\bar P$ and $\bar Q$ are the weight and root lattices for $\bar \g$. 

\item $\widetilde W= \overline{W} \ltimes \bar P$ is the extended affine Weyl group. For $\lambda \in \bar P$ we write $t_\lambda$ for $(e, \lambda) \in  \widetilde W.$ We also can write $\widetilde W = \mathcal{T} \ltimes W$, where $\mathcal{T}$ is a subgroup of the group of diagram automorphisms of $\g$.
 
 \item $(,)$ is the unique non-degenerate, Weyl-group invariant bilinear form on the weight space normalized so that $(\lambda, \delta)= \langle c, \lambda \rangle$, where $\delta$ is the minimal imaginary root and $c$ is the canonical central element of $\fg$. Write $\delta = \sum a_i h_i$ and $c = \sum a^\vee_i h_i$, where $\{ h_i \}_{i \in \bar I}$ is the set of simple coroots. Then we have $(\alpha_i, \alpha_j) = a_i^\vee a_i^{-1} a_{i,j}$. See \cite[Section 2.1]{BN}.

\item Following \cite[Section 2.2]{BN}, for each $i\in \bar I$, define $d_i = \max \{ 1, \frac{(\alpha_i,\alpha_i)}{2} \}$. When $\fg$ is untwisted, all the $d_i$ are equal to one. 
Define $q_s = q^{1/d}$, where $d = \max \{d_i \}$. For each $i\in I$, we define $q_i = q^{\frac{(\alpha_i,\alpha_i)}{2}}$, which is a power of $q_s$.

\item Define $r_i \in \{ 1, 2 \}$ by $r_i=1$ unless $\fg=A^{(n)}_{2n}$ and $i=n$, in which case $r_i = 2$.

\item $\U$ is the quantized universal enveloping algebra for $\g$, which is an algebra over $\QQ(q_s)$. The algebra $\U$ is generated by $\{ E_i \mid i \in I\}$ and $\{ F_i \mid i \in I\}$ (the standard Chevalley generators) and elements of the form $q^h$ for $h \in \frac{1}{d}P^*$ ($P^*$ is the coweight lattice for $\g$). 

We follow the conventions in \cite[Section 2.2]{BN}. 
For example, we have the following relation
\begin{align*}
  E_i F_i - F_i E_i = \frac{q_i^{h_i} - q_i^{-h_i}}{q_i - q_i^{-1}}
\end{align*}
where $\{h_i\}_{i \in I}$ are the simple coroots. See \cite[Section 2.2]{BN} for the full set of relations.

\item $\Uplus$ is the subalgebra of $\U$ generated by $\{ E_i \mid i \in I\}$. 

\item Kashiwara's involution $*: \U \rightarrow \U$ is the algebra anti-involution of $\U$ which fixes all the Chevalley generators. We denote the result of applying $*$ to $x$ by $x^*$. Note that $*$ preserves $\Uplus$. 
  
\item Let us define the bar-involution on the field $\QQ(q_s)$ by $\overline{f(q_s)} = f(q_s^{-1})$. The algebra $\U$ has a bar-involution defined by requiring $\overline{1} = 1$, $\overline{E_i} = E_i$ and $\overline{F_i} = F_i$ for all $i \in I$, and $\overline{f \cdot x \cdot y} = \overline{f} \cdot \overline{x} \cdot \overline{y}$ for all $f \in \QQ(q_s)$ and $x,y \in \U$. 

\item $\cB$ is Lusztig's canonical basis (equivalently Kashiwara's global crystal basis) of $\Uplus$. 

\item $\cA$ is the ring of rational functions in ${\Bbb C}(q)$ which are regular at $q=\infty$, and $\mathcal{L}= \text{span}_\cA \cB$. Recall that $\cB$ descends to a basis of $\cL/q^{-1} \cL$.

\item $\e_i,\f_i$ are Kashiwara's crystal operators on $\cB$ (as a basis for $\cL/q^{-1} \cL$).

\item The algebra $\Uplus$ has an ``integral'' form $\Uplus_{\mathscr{A}}$ defined over the ring $\mathscr{A} = \ZZ[q_s,q_s^{-1}]$. It is the $\mathscr{A} = \ZZ[q_s,q_s^{-1}]$-subalgebra of $\Uplus$ generated by the quantum-divided power vectors $E_i^{(n)} = \frac{E_i^n}{[n]_{q_i}!}$.

\end{itemize}

\subsection{Convex orders}

Recall that $\Delta_+^{min}$ is the set of positive real roots along with the minimal imaginary root $\delta$.

\begin{Definition} \label{def:convex}
  A {\bf convex order} is a total order $\prec$ on $\Delta_+^{min}$ (or more generally on any set of vectors) such that, given $S, S' \subset \Delta_+^{min}$ with $S \cup S' = \Delta_+^{min}$ and $ \alpha' \prec \alpha$ for all $\alpha \in S, \alpha' \in S'$, the convex cones $\text{span}_{{\mathbb R}_{\geq 0}} S$ and $\text{span}_{{\mathbb R}_{\geq 0}} S'$ intersect only at the origin. 
\end{Definition}

\begin{Remark}
People often use the notion of a clos order instead of Definition \ref{def:convex}, where a
 total order $\prec$ on $\Delta_+^{min}$ is called clos if, whenever  $\alpha, \beta$ and $\alpha+\beta$ are all in  $\Delta_+$, $\alpha+\beta$ occurs between $\alpha$ and $\beta$. This is equivalent in finite and affine types as shown in e.g. \cite[Lemma 2.11]{BKT}.
\end{Remark}

The following is immediate from the classification of convex orders in \cite{Ito}.
\begin{Proposition} \label{prop:1ro} If every real root in $\Delta_+$ is finitely far from one end of the order, then there is 
an infinite expression $\cdots s_{-2} s_{-1}s_0 s_1 s_2 \cdots$ such that, defining 
$$\beta_i = 
\begin{cases} s_1 \cdots s_{i-1} \alpha_i \quad \text{ if } i \geq 0 \\
s_0 s_{-1} \cdots s_{i+1} \alpha_i \quad \text{ if } i \leq 0
\end{cases}$$
for each $i \in {\Bbb Z}$,
the order is
$\mbox{} \quad \beta_0 \prec \beta_{-1} \prec \cdots \prec \delta \prec \cdots \prec \beta_{2} \prec \beta_{1}.$ 
\end{Proposition}
Following \cite{Ito} we call orders that arise as in Proposition \ref{prop:1ro} {\bf one-row orders.}

\begin{Definition} \label{def:*order}
If $\prec$ is a convex order, denote by $\prec^*$ the reversed order. 
\end{Definition}

\begin{Definition}
Fix a convex order $\prec$ such that $\alpha_i$ is minimal (resp. maximal). Define a new convex order $\prec^{s_i}$ by 
$$\be \prec \gamma
\Leftrightarrow s_i\be \prec^{s_i}s_i\gamma \quad \text{ if } \quad  \be, \gamma\neq
\alpha_i$$
and such that $\alpha_i$ is maximal (resp. minimal) for $\prec^{s_i}$.
\end{Definition}

\begin{Definition}
Fix a convex order $\prec$. Let $(i_1, \cdots, i_N)$ be the reduced word  for which 
\begin{align}
  \label{eq:inversion-set-of-w-inv}
\alpha_{i_1} \prec s_{i_1} \alpha_{i_2} \prec \cdots \prec s_{i_1} \cdots s_{i_{N-1}} \alpha_{i_N}
\end{align}
are the first $N$ roots. Let $w = s_{i_1} \cdots s_{i_N}$, and define $\prec^w =(\cdots(( \prec^{s_{i_1}})^{s_{i_2}}) \cdots )^{s_{i_N}}$. 

Similarly, when the roots in \eqref{eq:inversion-set-of-w-inv} are the last roots of the order, we can also define $\prec^w $.

\end{Definition}

\begin{Lemma} \label{lem:near-end}
If $H$ is any hyperplane in $\mathfrak{h}^*$ which does not contain $\delta$,  then all but finitely many roots in $\Delta_+^{min}$ lie strictly on the same side of $H$.
\end{Lemma}

\begin{proof}
Impose any Euclidean metric on $\mathfrak{h}^*$. 
Recall that there is a finite set of vectors $v_1, \ldots v_n$ such that every root is $\Delta_+$ is of the form $n \delta + v_i$ for some $i$ and $n\geq 0$. If $\delta \not \in H$, then, for sufficiently large $N$, the vector $N \delta$ is farther from $H$ then the length of any $v_i$. So, for sufficiently large $N$ and all $i$, $N \delta + v_i$ is on the same side of $H$ as $\delta$. 
\end{proof}

\begin{Lemma} \label{lem:eo}
Fix any finite collection $\beta_1 \cdots \beta_N \in \Delta_+^{min} $ and any convex order $\prec$ on these roots. Then there is a one-row order on $\Delta_+^{min}$ that restricts to $\prec$. 
\end{Lemma}

\begin{proof} 
Define a family of hyperplanes $H_j$ for $ 0 \leq j \leq N$ through the origin as follows:
\begin{itemize}
\item $H_0$ is any hyperplane such that  all roots $\Delta_+^{min}$ are on the same side. 

\item For $0<j<N$, $H_j$ separates the convex hull of $\beta_j$ along with all roots in $\Delta_+^{min}$ on the same side of $H_{j-1}$ as $\beta_{j-1}$ from the convex hull of $\{ \beta_{j+1}, \ldots, \beta_N\}$. 

\item $H_N$ also has all roots on the same side. 
\end{itemize}
This is possible because the order on the $\beta_j$ is convex. Choose an Euclidean inner product $(,)$ on $\mathfrak{h}^*$. For each $j$ with $0 \leq j < N$, let $v_j$ be the unit normal vector to $H_j$ satisfying $ (v_j, \beta_r) > 0$ for $r > j$. For
$0 \leq t \leq N$, 
defined $H_t$ to be the hyperplane orthogonal to $v_t = (1+j-t) v_j+  (t-j) v_{j+1}$, where $j \leq t \leq j+1$. 
Possibly perturbing the $v_j$, we can assume each $H_t$ contains at most one minimal root. Certainly, each root is on only one $H_t$. Let $\prec'$ be the order on $\Delta_+^{min}$ given by the $t$ associated to the root as above. This is convex because every initial and final segment is given by the roots on one side of some hyperplane, and certainly restricts to $\prec$.  

Any real root $\beta$ occurs on some $H_t$, and $\delta$ is not on this hyperplane, so, by Lemma \ref{lem:near-end}, $\beta$ is finitely far from one of the end of the convex order. Hence $\prec'$ is one-row. 
\end{proof}

Fix a convex order $\prec$. If $\g\neq A^{(2)}_{2n}$, then there is a unique $ \bar w \in \bar W$ such that $\alpha \succ \delta$ if and only if $\bar \alpha \in \bar w \bar \Delta $. If  $\g=A^{(2)}_{2n}$, then $\alpha \succ \delta$ if and only if $\bar \alpha \in \bar  w \bar \Delta$ or $2 \bar \alpha \in \bar w \bar \Delta$. See \cite[Section 3.4]{TW} for further discussion.  

\begin{Definition} 
The {\bf coarse type} of $\prec$ is the $\bar w$ from above. We call the coarse type corresponding to the identity element $\bar e$ the {\bf standard coarse type}.
\end{Definition}

\subsection{Crystal basis and canonical basis} 
Recall that  $\mathcal{A}$ is the ring of rational functions in the variable $q_s$ that are regular at $q_s=\infty$. Kashiwara defines a certain $\mathcal{A}$ submodule $\cL = \mathcal{L}(-\infty)$ of $U_q^+(\mathfrak{g})$ called the crystal lattice, and a basis $B(-\infty)$ of $\mathcal{L}/q_s^{-1} \mathcal{L}$, 
 called the crystal basis. He also defines crystal operators $\tilde f_i$, which are partial permutations of $B(-\infty)$. 

\begin{Theorem} \label{th:defofB}
There is a unique basis $\B$ for $U_q^{+}(\mathfrak{g})_{{\Bbb Z}[q_s,q_s^{-1}]}$ such that
\begin{enumerate}

\item $\text{span}_\mathcal{A}(\B)$ is the crystal lattice $\mathcal{L}(-\infty)$. 

\item $\B+q_s^{-1} \mathcal{L}(-\infty)= B(-\infty).$

\item Every element of $\B$ is bar invariant. 

\end{enumerate}
\end{Theorem}

\begin{Definition}
The basis $\B$ from Theorem \ref{th:defofB} is called Kashiwara's global crystal basis or Lusztig's canonical basis.
\end{Definition}

\subsection{The combinatorial crystal}

\begin{Definition}\label{def:crystal} (see \cite[Section 7.2]{Kashiwara:1995}) A {\bf combinatorial crystal} is a set $B$ along with functions $\wt \colon B \to P$ (where $P$ is the weight
  lattice), and, for each $i \in I$, $\varepsilon_i, \varphi_i \colon B \to {\mathbb Z}$ and $\e_i, \f_i: B \rightarrow B \sqcup \{ \emptyset \}$, such that
  \begin{enumerate}
  \item $\varphi_i(b) = \varepsilon_i(b) + \langle \wt(b), \alpha_i^\vee \rangle$.
  \item $\e_i$ increases $\varphi_i$ by 1, decreases $\eps_i$ by 1 and increases $\wt$ by $\alpha_i$.
  \item $\f_i b = b'$ if and only if $\e_i b' = b$.
  \end{enumerate}
We often denote a combinatorial crystal simply by $B$, suppressing the other data.
\end{Definition}

\begin{Remark}
Sometimes $\varphi_i,\varepsilon_i$ are allowed to be $-\infty$, but we do not need that case.
\end{Remark}

\begin{Definition}
\label{def:lwcrystal}
A combinatorial crystal  is called {\bf lowest weight} if it has a distinguished element $b_-$ (the lowest weight element) such that 
\begin{enumerate}
\item $b_-$ can be reached from any $b \in B$ by applying a sequence of $\f_i$  for various $i \in I$. 
\item For all $b \in B$ and all $i \in I$, $\varphi_i(b) = \max \{ n : \f_i^n(b) \neq \emptyset \}$.
\end{enumerate}
  \end{Definition}
\noindent For a lowest weight combinatorial crystal, $\varphi_i, \varepsilon_i$ and $ \wt$ are determined by the
$\f_i$ and $\wt(b_-)$.

The following is essentially the characterization of $B(-\infty)$ due to Kashiwara and Saito \cite{KS:1997}, but has been modified to make the ordinary and $*$ crystal operators play more symmetric roles. See \cite{TW} for this exact statement.  

\begin{Proposition} \label{prop:comb-char}
Fix a set $B$ along with two sets of operators $\{\tilde e_i, \tilde f_i\}$ and $\{ \tilde e_i^*, \tilde f_i^* \}$  such that  $(B, \e_i, \f_i)$ and $(B, \e_i^*, \f_i^*)$ are both lowest weight combinatorial crystals with the same lowest weight element $b_-$, where the other data is determined by $\wt(b_-)=0$. Assume that, for all $i \neq j \in I$ and all $b \in B$,
\begin{enumerate}

\item \label{ccc0} $\e_i(b), \e_i^*(b) \neq 0$.

\item \label{ccc1} $\e_i^*\e_j(b)= \e_j\e_i^*( b)$.

\item \label{ccc2} $\varphi_i(b)+\varphi_i^*(b)- \langle \wt(b),
  \alpha_i^\vee \rangle\geq0$.

\item \label{ccc3} If $\varphi_i(b)+\varphi_i^*(b)- \langle \wt(b), \alpha_i^\vee \rangle =0$ then $\e_i(b) = \e_i^*(b)$.

\item \label{ccc4} If $\varphi_i(b)+\varphi_i^*(b)- \langle \wt(b), \alpha_i^\vee \rangle \geq 1$ then $\varphi_i^*(\e_i(b))= \varphi_i^*(b)$ and  $\varphi_i(e^*_i(b))= \varphi_i(b)$.

\item \label{ccc5} If $\varphi_i(b)+\varphi_i^*(b)- \langle \wt(b), \alpha_i^\vee \rangle \geq 2$ then  $\e_i \e_i^*(b) = \e_i^*\e_i(b)$.

\end{enumerate}
Then $(B, \e_i, \f_i) \simeq (B, \e_i^*, \f_i^*) \simeq B(-\infty)$, and
$\e_i^*= *\e_i *, \f_i^*=*\f_i*$, where $*$ is Kashiwara's
involution. \qed
\end{Proposition}

We can understand Proposition \ref{prop:comb-char} as follows. 
For any $i \in I$, and any $b \in B(-\infty)$, the subset of $B(-\infty)$ generated by the operators $\e_i, \f_i, \e_i^*, \f_i^*$ is of the form:
\vspace{0.15cm}

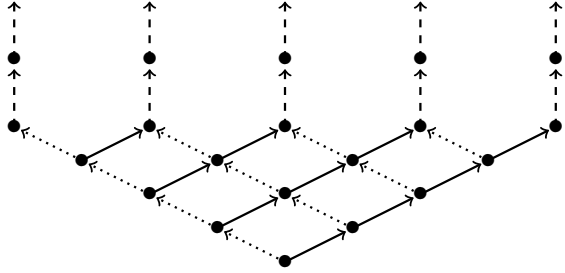
\begin{equation*} \label{ii*-pic}
\setlength{\unitlength}{0.15cm}
\begin{tikzpicture}[xscale=0.45,yscale=-0.45, line width = 0.03cm]

\draw node at (10,5) {$\bullet$};

\draw node at (8,4) {$\bullet$};
\draw node at (12,4) {$\bullet$};
\draw node at (6,3) {$\bullet$};
\draw node at (10,3) {$\bullet$};
\draw node at (14,3) {$\bullet$};

\draw node at (4,2) {$\bullet$};
\draw node at (8,2) {$\bullet$};
\draw node at (12,2) {$\bullet$};
\draw node at (16,2) {$\bullet$};

\draw node at (2,1) {$\bullet$};
\draw node at (6,1) {$\bullet$};
\draw node at (10,1) {$\bullet$};
\draw node at (14,1) {$\bullet$};
\draw node at (18,1) {$\bullet$};

\draw node at (2,-1) {$\bullet$};
\draw node at (6,-1) {$\bullet$};
\draw node at (10,-1) {$\bullet$};
\draw node at (14,-1) {$\bullet$};
\draw node at (18,-1) {$\bullet$};

\draw [->, dotted] (10,5)--(8.2,4.1); 
\draw [->, dotted] (8,4)--(6.2,3.1); 
\draw [->, dotted] (12,4)--(10.2,3.1); 
\draw [->, dotted] (6,3)--(4.2,2.1); 
\draw [->, dotted] (10,3)--(8.2,2.1); 
\draw [->, dotted] (14,3)--(12.2,2.1); 

\draw [->, dotted] (4,2)--(2.2,1.1); 
\draw [->, dotted] (8,2)--(6.2,1.1); 
\draw [->, dotted] (12,2)--(10.2,1.1); 
\draw [->, dotted] (16,2)--(14.2,1.1); 
\draw [->] (10,5)--(11.8,4.1); 

\draw [->] (8,4)--(9.8,3.1); 
\draw [->] (12,4)--(13.8,3.1); 
\draw [->] (6,3)--(7.8,2.1); 
\draw [->] (10,3)--(11.8,2.1); 
\draw [->] (14,3)--(15.8,2.1); 
\draw [->] (4,2)--(5.8,1.1); 
\draw [->] (8,2)--(9.8,1.1); 
\draw [->] (12,2)--(13.8,1.1); 
\draw [->] (16,2)--(17.8,1.1); 

\draw [->, dashed] (2,1) --(2,-0.7);
\draw [->, dashed] (6,1) --(6,-0.7);
\draw [->, dashed] (10,1) --(10,-0.7);
\draw [->, dashed] (14,1) --(14,-0.7);
\draw [->, dashed] (18,1) --(18,-0.7);
\draw [->, dashed] (2,-1) --(2,-2.7);
\draw [->, dashed] (6,-1) --(6,-2.7); 
\draw [->, dashed] (10,-1) --(10,-2.7);
\draw [->, dashed] (14,-1) --(14,-2.7);
\draw [->, dashed] (18,-1) --(18,-2.7);

\end{tikzpicture}
\end{equation*}

\noindent where the solid or dashed arrows show the action of
$\e_i$, and the dotted or dashed arrows show the action of $\e_i^*$.  Here the width of the diagram at the top is $-\langle \wt(b_v), \alpha_i^\vee \rangle$, where $b_v$ is the bottom vertex (in the example above the width is 4).

\subsection{The braid group action}

For $i \in I$ we consider Lustzig's braid group operator $T_i$ (denoted $T^{''}_{i,1}$ in \cite[Section 37]{Lus}) which is the algebra automorphism of $U_q(\g)$ defined by 
\begin{align*}
& T_i ( E_i) =  -F_i q_i^{h_i}\\
& T_i ( F_i) =  -q_i^{-h_i} E_i \\
& T_i ( E_j) = \sum_{r+s = - a_{i,j}} (-1)^r q_i^{-r} E_i^{(s)} E_j E_i^{(r)}  \text{ for } i \neq j \\ 
& T_i ( F_j) = \sum_{r+s = - a_{i,j}} (-1)^r q_i^{r} F_i^{(r)} F_j F_i^{(s)}  \text{ for } i \neq j \\ 
& T_i( q^{h}) = q^{s_i(h)}
\end{align*}
These satisfy the braid relations so, for $w \in W$, we can unambiguously define
\begin{align}
  \label{eq:2}
T_w= T_{i_k} \cdots T_{i_1}
\end{align}
for any reduced expression $w=s_{i_k} \cdots s_{i_1}$.

\begin{Proposition}{\cite[Proposition 3.4.7]{Sai}} \label{prop:Saito}
Suppose $P \in \mathcal{L}$ specializes to a crystal basis element $b$ modulo $q^{-1}$, and that $T_i(P) \in \U^+$. Then $T_i(P) \in \mathcal{L}$ and specializes to the crystal basis element $\sigma_i(b)$ modulo $q^{-1}$, where 
\begin{align*}
\sigma_i(b) =   ( \e_i)^{\epsilon^*_i(b)} (\f_i^*)^{\varphi^*_i(b)} b.
\end{align*}
\end{Proposition}

\begin{Remark}
Proposition \ref{prop:Saito} is incorrectly stated in \cite[Theorem 4.13]{MT}. Relatedly, the proof of \cite[Corollary 4.14]{MT} is incorrect. This can be fixed by noticing that the correct form of Proposition \ref{prop:Saito} immediately gives, in that paper's notation,
\begin{align}
  \label{eq:55}
 L( \cc \circ s_i, i)  \equiv  \sigma_i L( \cc, i-1) \mod \cL/q^{-1} \cL,
\end{align}
from which \cite[Corollary 4.14]{MT} is immediate.
\end{Remark}

\section{Characterization of affine MV polytopes}

\subsection{Lusztig data} The definitions here are based on \cite{TW}. 

\begin{Definition}
A {\bf multipartition} $\underline\lambda = (\lambda^{(1)}, \cdots, \lambda^{(n)})$ is a collection of $n$ partitions indexed by $\bar I$.
The {\bf weight} of $\underline \lambda$ is
  \begin{align}
    \label{eq:39}
 \wt(\underline \lambda) =   \sum_{i\in \bar I} d_i |\lambda^{(i)}|.
  \end{align}
\end{Definition}

\begin{Definition}
  A {\bf Lusztig datum} $\cc$ is a collection of non-negative integers $\{ \cc_{\beta} \}$ indexed by positive real roots $\beta$ along with a multipartition $\cc_\delta$ such that $\cc_{\beta} = 0$ for almost all real roots $\beta$. The weight of the Lusztig datum $\cc$ is given by:
  \begin{align}
    \label{eq:51}
 \wt( \cc) = \sum_{\beta} \cc_{\beta} \cdot \beta + \wt(\cc_{\delta}) \cdot \delta  
  \end{align}
\end{Definition}

If $\cc$ is a Lusztig datum such that $\cc_{\beta} = 0$ for all real roots, then we say that $\cc$ is {\bf purely imaginary}. In this case, let $\underline \lambda  = \cc_{\delta}$. We will often abuse notation and write  $\cc = \underline \lambda$, i.e. we will just write multipartitions to mean the corresponding purely imaginary Lusztig data.

\begin{Definition}
Fix a convex order $\prec$. A root $\alpha$ is called {\bf accessible} if it is finitely far from one end of the order. 
\end{Definition}

\begin{Remark}
For one row orders every real root is accessible.  
\end{Remark}

\begin{Definition}\label{ctLusztig}
Fix a convex order $\prec$ and an accessible real root $\alpha$. For each $b \in B(-\infty)$, define an integer $\coa_{\alpha}^\prec(b)$ by setting $\coa^\prec_{\alpha_i}(b)=\varphi_i(b)$ if $\alpha_i$ is minimal for $\prec$,
  and \[\coa^\prec_{\alpha}(b)=\coa_{s_i\alpha}^{\prec^{s_i}}(\sigma_i(f_i^{\varphi_i(b)}b))\]
  for all other roots finitely far from the beginning of the order.
Similarly, define 
  $\coa^{\prec}_{\alpha_i}(b)=\varphi^*_i(b)$ if $\alpha_i$ is maximal for $\prec$,
  and for other roots finitely far from the end of the order, define
  \[\coa^{\prec}_{\alpha}(b)=\coa_{s_i\alpha}^{\prec^{s_i}}(\sigma_i^*((f^*_i)^{\varphi^*_i(b)}b)).\]
We call $\{ \coa^\prec_\alpha(b) \}$ the {\bf crystal-theoretic real Lusztig data} for $b$ with respect to $\prec$.
\end{Definition}

\begin{Remark}
We could actually define crystal theoretic Lusztig data for any convex order. Given an order $\prec$ that is not one-row, choose a one row order $\prec'$ that orders all roots of height less than or equal to $\langle \wt(b), \rho \rangle$ in the same way as $\prec$ (see Lemma \ref{lem:eo}). Then the crystal theoretic Lusztig data for $\prec$ is defined to be the crystal theoretic Lusztig data for $\prec'$. 
The fact that this is well-defined will follow from the fact that crystal-theoretic Lusztig data agrees with Lusztig data of PBW basis vectors (Proposition \ref{prop:cCS}) and the fact that the Lusztig data of PBW basis vectors behave well when approximating arbitrary orders by one-row orders (Corollary \ref{cor:onS+}).
\end{Remark}

\subsection{Affine MV polytopes}

An affine pseudo-Weyl polytope $P$ is a polytope in $\fh^*$ (considered up to translation) such that every edge is an integer multiple of a root (after translating one end of the edge to the origin).

Each point $\mu$ in the weight space of $\overline\fg$ defines a linear functional on $\fh^*$, and hence defines a face $P^\mu$ of $P$ by taking the points where this functional achieves its minimum, and this face is always parallel to the imaginary root $\delta$. If $\mu$ is in the interior of some Weyl chamber $C$, then this will be a line $P^C$ parallel to $\delta$, and that line will not depend on the precise point chosen. If $\mu$ is a chamber weight $\bar w \bar \omega_i$ for some $\bar w \in \overline{W}$, then $P^\mu$ will usually be a face of codimension 1. 

\begin{Definition} \label{def:edge-decorated}
An {\bf edge-decorated affine pseudo-Weyl polytope} 
is an affine pseudo-Weyl polytope $P$ along with a choice of multipartition $\underline\lambda  = (\lambda^{(1)}, \cdots, \lambda^{(n)})$ for each $\bar w \in \bar W$ such that the length of the edge $P^{\bar w C_+}$ is equal to $\wt(\underline \lambda)$, where $C_+$ is the dominant chamber. 

\end{Definition}

\begin{Definition}
A {\bf facet-decorated affine pseudo-Weyl polytope} is an affine pseudo-Weyl polytope along with a choice of a partition $\pi^\gamma$ for each chamber weight $\gamma$ of $\overline\fg$, which satisfies the condition that, for each $w \in W$, the length of the edge $P^{w C_+}$ is equal to $\sum_{i \in \bar I} d_{i} |\pi^{w \omega_i}|$. 

\end{Definition}

A facet-decorated affine pseudo-Weyl polytope gives rise to an edge-decorated affine pseudo-Weyl polytope where, for each  $\overline w \in \bar W$, we define the associated multipartition $\underline\lambda  = (\lambda^{(1)}, \cdots, \lambda^{(n)})$ by $\lambda^{(i)} = \pi^{\overline w \bar \omega_i}$. We say that an edge-decorated affine pseudo-Weyl polytope is induced from a face-decorated polytope if it arises in this way. 

\begin{Definition}
For a pseudo-Weyl polytope $P$, let $\mu_0(P)$ be the vertex of $P$ such that $\langle \mu_0(P), \rho^\vee \rangle$ is lowest, and $\mu^0(P)$ the vertex where this is highest (these are vertices as for all roots $\langle \alpha, \rho^\vee \rangle \neq 0$). 
\end{Definition}

\begin{Lemma} \label{lem:ispath} \cite[Lemma 1.20]{TW}
Fix a pseudo-Weyl polytope $P$ and a convex order $\prec$ on $\Delta_+^{min}$. There is a unique path $P^\prec$ through the 1-skeleton of $P$ from $\mu_0(P)$ to $\mu^0(P)$ which passes through at most one edge parallel to each root, and these appear in decreasing order according to $\prec$ as one travels from $\mu_0(P)$ to $\mu^0(P)$.
\end{Lemma}

\begin{Definition} \label{def:geom-LD}
Fix an edge-decorated pseudo-Weyl polytope $P$ and a convex order $\prec$. For each positive real root $\beta$, define $\cc^\prec_\beta(P)$ to be the unique non-negative number such that the edge in $P^\prec$ parallel to $\beta$ is a translate of $\cc^\prec_\beta(P) \beta$. We call the collection $\{ \cc^\prec_\beta(P) \}$ the {\bf polytopal real Lusztig data} of $P$ with respect to $\prec$. 
\end{Definition}

\subsection{Characterization}

The following characterization of affine MV polytopes is a straightforward extension of \cite[Theorem 3.11]{MT} and \cite[Proposition 1.24]{TW}.

\begin{Theorem} \label{th:MV-def}
There is a unique map $b \rightarrow P_b$ from $B(-\infty)$ to edge-decorated pseudo-Weyl polytopes such that $P_{b_-}$ is a point and the following conditions are satisfied.  

\begin{enumerate}

\item[(C)] 
\begin{itemize} 

\item If $\alpha_i$ is minimal for $\prec$ and $\f_i(b) \neq 0$, then $\cc^\prec_{\alpha_i}(P_{\f_i(b)})= \cc^\prec_{\alpha_i}(P_{b})-1$, and $\cc^\prec_{\beta}(P_{\f_i(b)})= \cc^\prec_{\beta}(P_{b})$ for all other positive roots $\beta$.

\item If $\alpha_i$ is maximal for $\prec$ and  $\f^*_i(b) \neq 0$, then $\cc^\prec_{\alpha_i}(P_{\f^*_i(b)})= \cc^\prec_{\alpha_i}(P_{b})-1$, and $\cc^\prec_{\beta}(P_{\f^*_i(b)})= \cc^\prec_{\beta}(P_{b})$ for all other positive roots $\beta$.

\end{itemize}

\item[(S)]
\begin{itemize}

\item If $\f_i(b)=0$ and $\alpha_i$ is minimal for $\prec$, then, for all $\alpha\neq \alpha_i$, $\cc^\prec_\alpha(P_b)= \cc^{\prec^{s_i}}_{s_i \alpha}(P_{\sigma_ib})$. 

\item If $\f^*_i(b)=0$ and $\alpha_i$ is maximal for $\prec$, then, for all $\alpha \neq \alpha_i$, $\cc^\prec_\alpha(P_b)= \cc^{\prec^{s_i}}_{s_i \alpha}(P_{\sigma_i^*b})$. 

\end{itemize}

\item[(I)] 
Let $i \in \bar I$, and Let $(\lambda^{(1)}, \ldots, \lambda^{(n)})$ be a multipartition with $\lambda^{(i)} \neq \emptyset$. 
let $\prec$ be a convex order of course type $\overline e$ and $\alpha_i$ is minimal for $\prec$. 
Suppose $b \in B(-\infty)$ satisfies $\cc_\beta^{\preci}(P_b)=0$ for all real roots $\beta$, and suppose $\cc_\delta^{\preci}(P_b)= (\lambda^{(1)}, \ldots, \lambda^{(n)})$. Then we have
\begin{itemize}
\item 
$\cc_{\alpha_i}^{\prec}(P_b)= r_i \cdot \lambda^{(i)}_1 $ and $\cc_{d_i\delta-r_i\alpha_i}^{\prec}(P_b)= \lambda^{(i)}_1$, and for all other real roots $\beta$, $\cc_\beta^{\prec}(P_b)=0$.

\item Write $\cc^{\prec}_\delta(P_b) = (\mu^{(1)}, \ldots, \mu^{(n)})$. Then $\mu^{(i)} = \lambda^{(i)} \backslash  \lambda^{(i)} _1$ (i.e. the largest part is removed) and  $ \mu^{(j)}=\lambda^{(j)}$ for $j \neq i$.   
\end{itemize}
 \noindent Conversely, if $\cc^\prec(P_b)$ is of the form above, then $\cc^\preci(P_b)$  must also be as above.

\end{enumerate}
Furthermore, every $P_b$ in the image is induced from a facet-decorated polytope, which agrees with the MV polytope $MV_b$ from \cite{TW}. 
\end{Theorem}

\begin{Remark} \label{rem:GCdata}
Conditions (C) and (S) immediately imply that, for any one-row convex order $\prec$, and any real root $\beta$, $\cc^\prec_\beta(P_b)= \coa^\prec_\beta(b)$. That is, the polytopal real Lusztig data agrees with the crystal theoretic Lusztig data. 
\end{Remark}

\begin{proof}
Such a map is constructed in \cite{TW}, where it is also shown that the image consists only of face-decorated polytopes. It remains to prove uniqueness. 
 
Suppose we have two such maps $b \mapsto P_b$ and $b \mapsto Q_b$. It suffices to check that, for each $b \in B(-\infty)$, $P_b$ and $Q_b$ have the same Lusztig data for every one-row convex order. 
Proceed by induction on $\langle \wt(b), \rho^\vee \rangle$, the base case $b=b_-$ being obvious. 
Conditions (C) and (S) guarantee that for any $b \in B(-\infty)$, and any real $\beta$, $\cc_\beta^\prec(P_b)=\cc_\beta^\prec(Q_b)$, since both agree with the real crystal theoretic Lusztig data.

Fix $b \in B(-\infty)$, and suppose the uniqueness statement is true for all $b'$ such that $\langle \wt(b'), \rho^\vee \rangle < \langle \wt(b), \rho^\vee \rangle$. Fix a one-row order $\prec$. Consider the case where $P_b$ has some non-zero real Lusztig data with respect to $\prec$. Assume that $\cc_\beta^\prec(P_b) \neq 0$ for some real root $\beta \prec \delta$; the case $\beta \succ \delta$ is similar.   Let $\beta$ be the minimal such root, and set $n = \coa^\succ_\beta(b)$. 

Since $\prec$ is a one-row order, there is a reduced word $(i_1, \cdots, i_N)$ such that $ \alpha_{i_1} \prec s_{i_1}\alpha_{i_2} \prec \cdots \prec s_{i_1} \cdots s_{i_{N-1}} \alpha_{i_N} = \beta$ are the first $N$ real roots of the order.  Then consider the crystal element $b^\prime = \sigma_{i_1}^* \cdots \sigma_{i_{N-1}}^* \f_{i_N}^n \sigma_{i_{N-1}} \cdots \sigma_{i_1}b$. From this formula, we compute that $b^\prime \neq 0$,  $\wt (b^\prime) = \wt (b) - n \cdot \beta$, and $b =  \sigma_{i_1}^* \cdots \sigma_{i_{N-1}}^* \e_{i_N}^n \sigma_{i_{N-1}} \cdots \sigma_{i_1}b^{\prime}$. By induction we have $P_{b^\prime} = Q_{b^\prime}$, and in particular $\cc^{\prec}(P_{b^\prime} ) = \cc^{\prec}(Q_{b^\prime} )$.  

By applying condition (S) $N-1$ times, we have 
\[\cc^{\prec^{s_{i_1} \cdots s_{i_{N-1}}}}( P_ {\sigma_{i_{N-1}} \cdots \sigma_{i_1}b^{\prime}}) = \cc^{\prec^{s_{i_1} \cdots s_{i_{N-1}}}}( Q_ {\sigma_{i_{N-1}} \cdots \sigma_{i_1}b^{\prime}}).\] 
Applying condition (C), we have 
\[\cc^{\prec^{s_{i_1} \cdots s_{i_{N-1}}}}( P_ { \e_{i_N}^n\sigma_{i_{N-1}} \cdots \sigma_{i_1}b^{\prime}}) = \cc^{\prec^{s_{i_1} \cdots s_{i_{N-1}}}}( Q_ { \e_{i_N}^n\sigma_{i_{N-1}} \cdots \sigma_{i_1}b^{\prime}}).\] 
Applying condition (S) $N-1$ times again, 
\[\cc^{\prec}(P_{\sigma_{i_1}^* \cdots \sigma_{i_{N-1}}^* \e_{i_N}^n \sigma_{i_{N-1}} \cdots \sigma_{i_1}b^\prime} ) = \cc^{\prec}(Q_{\sigma_{i_1}^* \cdots \sigma_{i_{N-1}}^* \e_{i_N}^n \sigma_{i_{N-1}} \cdots \sigma_{i_1}b^\prime} ),\] 
But $b=\sigma_{i_1}^* \cdots \sigma_{i_{N-1}}^* \e_{i_N}^n \sigma_{i_{N-1}} \cdots \sigma_{i_1}b^\prime$, so $\cc^{\prec}(P_{b} ) = \cc^{\prec}(Q_{b} )$.

Now consider the case where $\cc_\beta^\prec(P_b)=0$ for all real roots $\beta$. We must have $\cc_\beta^\prec(Q_b)=0$ for all real $\beta$ by comparison to the crystal theoretic Lusztig data.   Let $(\lambda^{(1)}, \cdots , \lambda^{(n)}) = \cc^\prec_\delta (P_b)$. Let $i \in \bar I$ be such that $\lambda^{(i)}$ is a non-empty partition. Applying condition (S) if necessary, we can assume that  the course type of $\prec$ is $\overline e$. Because the crystal-theoretic Lusztig data agrees with the polytopal Lusztig data, we may further assume $\alpha_i$ is minimal for $\prec$.  Applying further $\sigma_i$, we are in the situation of  condition (I), i.e. we have $\cc^\preci_\delta(P_{\sigma_i(b)}) = \underline \lambda$ and $\cc^\preci_\beta(P_{\sigma_i(b)}) = 0$ for all real roots $\beta$. 
Applying condition (I), we have $\frac{1}{r_i} \cc_{\alpha_i}^{\prec}(P_{\sigma_i(b)})=\cc_{d_i \delta-r_i\alpha_i}^{\prec}(P_{\sigma_i(b)})= n$ for the positive integer $n = \lambda^{(i)}_1$,  $\cc^{\prec}_{\beta}(P_{\sigma_i(b)}) = 0$ for all other real roots $\beta$, and $\cc^{\prec}_\delta(P_{\sigma_i(b)}) = (\lambda^{(1)}, \cdots, \lambda^{(i)}\backslash \lambda^{(i)}_1, \cdots, \lambda^{(n)})$. By comparison with the crystal theoretic Lusztig data, we have the same formulas for $\cc_\beta^\prec(Q_{\sigma_i(b)})$ for all real roots $\beta$.  

Let $b^\prime = \f_i^{r_i n} \sigma_i(b)$. By induction $P_{b^\prime} = Q_{b^\prime}$. Using condition (C), we have $\cc^{\prec}_\delta(Q_{b^\prime}) = \cc^{\prec}_\delta(P_{b^\prime}) = (\lambda^{(1)}, \cdots, \lambda^{(i)}\backslash \lambda^{(i)}_1, \cdots, \lambda^{(n)})$. Using condition (C) again, we see that  $\cc^{\prec}_\delta(Q_{\sigma_i(b)})= (\lambda^{(1)}, \cdots, \lambda^{(i)}\backslash \lambda^{(i)}_1, \cdots, \lambda^{(n)})$. 
In particular, we see that $\cc^{\prec}(Q_{\sigma_i(b)}) = \cc^{\prec}(P_{\sigma_i(b)})$. Using condition (I), we conclude $\cc^{\preci}(Q_{\sigma_i(b)})$ is purely imaginary, and  $\cc^{\preci}_\delta(Q_{\sigma_i(b)}) = (\lambda^{(1)}, \cdots , \lambda^{(n)}) $, that is, $\cc^{\preci}(Q_{\sigma_i(b)})=\cc^{\preci}(P_{\sigma_i(b)})$. Applying (S) one last time, we have  $\cc^{\prec}(Q_{b})=\cc^{\prec}(P_{b})$.
\end{proof}

\begin{Remark}
We will construct a map satisfying Theorem \ref{th:MV-def} later in this paper (Theorem \ref{thm:PBWMV}), which can be used to complete the proof working only with affine PBW bases, so without appealing to \cite{TW}.
\end{Remark}

\begin{Remark}
For every 2-face $F$ of a decorated pseudo-Weyl polytope, the roots parallel to $F$ form a rank 2 sub-root
system $\Delta_F$ of either finite or affine type. Up to a small subtlety for faces parallel to $\delta$, affine MV polytopes are exactly those pseudo-Weyl polytopes such that every 2-face is an MV polytope for the corresponding rank 2 (finite or affine) root system. See \cite[Theorem B]{TW}.
\end{Remark}

\begin{Definition}
The decorated polytope $MV_b$ from Theorem \ref{th:MV-def} is called the MV polytope for $b$. 
\end{Definition}

The proof of Theorem \ref{th:MV-def} also yields the following, which we will need later. 

\begin{Proposition} \label{rem:partial-char} 
 Fix a multipartition $\underline{\lambda}$. Assume all the conditions of the Theorem \ref{th:MV-def} hold except that (I) is only known to hold for all multipartitions $\underline{\mu}$ with $\wt(\underline{\mu}) < \wt(\underline{\lambda})$ or $\underline{\mu}=\underline{\lambda}$.  Then for any convex order $\prec$, we have $ \cc^\prec(MV_b) = \cc^\prec(P_b)$  for all $b$ with $\wt(\cc^\prec_\delta(P_b)) < \wt(\underline\lambda)$ or $\cc^\prec_\delta(P_b)= \underline\lambda$. \qed
\end{Proposition}

\section{Affine PBW bases}

We now a PBW basis for any convex order on $\Delta_+^{min}$. 

\subsection{Real root vectors}

\begin{Definition}
Fix a one-row order on $\Delta^{min}_+$ and consider the corresponding infinite reduced expression $\cdots s_{i_{-2}} s_{i_{-1}} s_{i_0} s_{i_1} s_{i_2} \cdots$ from Proposition \ref{prop:1ro}. Define the corresponding {\bf real root vectors} by 
$$E^\prec_{\beta_k} = 
\begin{cases} T_{i_1} \cdots T_{i_{k-1}} E_{i_k}  \quad \text{ if } k \geq 0 \\
T_{i_{0}}^{-1}T_{i_{-1}}^{-1} \cdots T_{i_{k+1}}^{-1} E_{i_k} \quad \text{ if } k \leq 0.
\end{cases}$$
When the chosen order is clear we often leave off the superscript $\prec$. 
\end{Definition}

\begin{Proposition} \label{prop:equalE}
Fix $\beta \in \Delta_+^{min}$. Let $M_\beta = \{ (\alpha_1, \alpha_2) \in \Delta_+^2 : \beta = \alpha_1 + \alpha_2 \}$. If $\prec, \prec'$ are two one-row convex orders whose restriction to all pairs in $M_\beta$ agree then $E_\beta^\prec= E_\beta^{\prec'}$.
\end{Proposition}
\begin{proof}
Let $S^+$ be the set of roots $\succ \beta$, and $S^-$ the set of roots $\prec \beta$. Define things similarly for $\prec'$. Choose hyperplanes $H$ and $H'$ containing $\beta$ and the origin, and separating the convex hulls of $S^+, S^-$ and $S'^+, S'^-$ respectively. Choose an Euclidean inner product $(,)$ on $\mathfrak{h}^*$. Let $v$ (resp. $v'$) be the unit normal vector to $H$ (resp. $H'$) satisfying $ (v, S+) > 0$ (resp. $(v',S'+) > 0$). For
$0 \leq t \leq N$, 
defined $H_t$ to be the hyperplane that is the orthocomplement of the vector 
$v_t = t \cdot v+ (1-t)\cdot  v'$ for $0 \leq t \leq 1$. By possibly perturbing $v, v'$, we can assume that the intersection of each $H_t$ with $\Delta^+$ is usually just $\beta$, and otherwise is contained in a 2 dimensional subspace, and that this happens at most countably many times. 

Each time one moves past a $t$ where $H_t$ contains a rank 2 root system the convex order changes by reversing the order of that rank 2 root system. If $\beta$ is not part of that root system $E_\beta$ certainly does not change. Otherwise, the condition in the statement implies that the system does not contain a pair $(\alpha_1, \alpha_2) \in M_\beta$, so $\beta$ must be a simple root in that system. Checking the rank two finite and affine root systems case by case one sees that $E_\beta$ does not change as you reverse the order. The lemma follows.
\end{proof}

\begin{Corollary} \label{cor:simpleroots}
For all simple roots $\alpha_i$ and all convex orders $\prec$, $E_{\alpha_i}^{\prec}=E_i$.
\end{Corollary}

\begin{proof}
If $\alpha_i$ is first in the order, $E_{\alpha_i}=E_i$ by definition. 
By Proposition \ref{prop:equalE}, since $\alpha_i$ cannot be written as a sum of any two positive roots, $E_{\alpha_i}$ is independent of the convex order, so this is true for all convex orders. 
\end{proof}

\begin{Corollary} \label{cor:onS+}
Fix $\prec $ and a root $\beta$.
Let $S^+$ be the set of roots $\succ \beta$, and $S^-$ the set of roots $\prec \beta$. 
The root vector $E^\prec_{\beta}$ only depends on the data of $S^+$, not the precise order on $S^+$ and $S^-$. 
\end{Corollary}
\begin{proof}
If $\beta=\beta'+\beta''$ then by convexity exactly one of $\beta', \beta''$ is in $S^+.$ Thus if the convex order is changed but in such a way that $S^+$ remains the same, they have not been reordered. The corollary follows from Proposition \ref{prop:equalE}.
\end{proof}

\begin{Definition} \label{prop:gen-real}
Fix a convex order $\prec$ and a real root $\beta$. By Lemma \ref{lem:eo} we can find a one-row order $\prec'$ that agrees with $\prec$ for all pairs of roots of depth at most the depth of $\beta$. Set $E_\beta^\prec= E_\beta^{\prec'}$ (this is well defined by Proposition \ref{prop:equalE}).
\end{Definition}

\begin{Corollary} \label{cor:can-real} Fix a coarse type.
If $\bar \alpha$ is a simple root for the corresponding positive system, then, for all $n \geq 0$, $E_{n\delta\pm\alpha}^\prec$ are the same for all orders $\prec$ with this coarse type.   
\end{Corollary}
\begin{proof}
Any expression $n \delta\pm\alpha= \beta_1+\beta_2$ must have $\beta_1 \leq \delta$ and $\beta_2 \geq \delta$ (or vice versa). The order of these two roots is the same for any convex order of the given coarse type. So the statement follows from Proposition \ref{prop:equalE}.
\end{proof}

\begin{Lemma} \label{lem:real*} 
For any convex order $\prec$ and any real root $\beta$, $(E_\beta^\prec)^* = E_\beta^{\prec^*}$, where on the left side $*$ is Kashiwara's involution, and, as in Definition \ref{def:*order}, $\prec^*$ is the reverse order of $\prec$.
\end{Lemma}

\begin{proof}
  This follows immediately from the definition and the fact that $T_i^{-1} = * \circ T_i \circ *$.
\end{proof}

\subsection{Imaginary root vectors}

In \cite{BCP}, the following pairwise commuting vectors of imaginary weight are introduced. 
\begin{align} \label{eq:imformula}
\ \tpsi_{i,kd_i} = E_{kd_i\delta - \alpha_i} E_{\alpha_i} - q_i^{-2}E_{\alpha_i} E_{kd_i\delta - \alpha_i}
\end{align}
The real root vectors are for the standard coarse type, and do not depend on which convex order of that type is used by Corollary \ref{cor:can-real}.
For every partition $\lambda$, they introduce imaginary root vectors $S^\lambda_i$ which are polynomials in the $\tpsi_{i,kd_i}$.
The polynomial $P_\lambda$ is recursively defined, and it's exact form is unimportant for our purposes; we remark however that $P_\lambda$ is the same in all cases except $A^{(n)}_{2n}$ and $i=n$. See \cite[Equation 3.8]{BN}. What is relevant here is that the ${S^{\lambda}_i}$ multiply exactly as Schur functions:
\begin{align}
  \label{eq:53}
  S^{\lambda}_i \cdot S^{\mu}_i = \sum_{\nu} c^{\lambda,\mu}_\nu S^{\nu}_i
\end{align}
where $c^{\lambda,\mu}_\nu$ is a Littlewood-Richardson coefficient,
and that $  \wt( S_{i}^{\bar w, \lambda}) = d_i |\lambda| \cdot \delta$
where $|\lambda|$ is the sum of the parts of the partition $\lambda$.

\begin{Definition} \label{def:im-root-vectors} 
Fix $\bar w \in \overline{W}$. Fix a convex order of coarse type $\bar w$. For each $i \in \bar I$ let $\gamma_i$ be the minimal affine root such that $\bar \gamma_i= \bar w \bar \alpha_i$. Define 
\begin{align*}
\IPsi{\bar w}{i,k}= E^{\bar{w}}_{k d_i\delta-\gamma_i}E^{\bar{w}}_{{\gamma_i}}  - q_i^{-2} E^{\bar{w}}_{{\gamma_i}}E^{\bar{w}}_{k d_i\delta-\gamma_i} 
\quad \text{  and } \quad
S_{i}^{\bar w, \lambda} = P_\lambda( \IPsi{\bar w}{i,1}, \cdots \IPsi{\bar w}{i,N(\lambda)} ).
\end{align*}
By Corolary \ref{cor:can-real} this does not depend on the choise of $\prec$. 
Notice that $\IPsi{\bar e}{i,k} = \tpsi_{i,kd_i}$.
\end{Definition}


\begin{Definition}
For any coarse type $\bar w$ and multipartition $\underline{\lambda} = (\lambda^{(1)},\cdots,\lambda^{(n)})$, let $S^{\bar w,\underline{\lambda}} = \prod_{i\in \bar I} S^{\bar w,\lambda^{(i)}}_i$.
For any convex order $\prec$ of coarse type $\bar w$, let $S^{\prec,\underline{\lambda}} = S^{\bar w,\underline{\lambda}}.$
\end{Definition}
We therefore have $\wt(S^{\prec,\underline{\lambda}} ) = \wt(\underline \lambda) \cdot \delta$. For simplicity write $S^\lambda_i = S^{\bar e, \lambda}_i$ and $S^{\underline \lambda} = S^{\bar e, \underline \lambda}_i$.
The following is originally due to Beck \cite{Beck1} in the untwisted affine case and was extended further by Damiani \cite{Dam} (see also \cite[Proposition 3.14]{BN}). 
\begin{Proposition}
  Let $i \in \bar I$. Except in the case of $A^{(2)}_{2n}$ and $i=n$, there is an explicitly defined injective algebra morphism 
  \begin{align}
   h_i: U_q(\asl_2) \rightarrow  U_q(\g).
  \end{align} 
In the case $A^{(2)}_{2n}$ and $i=n$, we instead have an injective algebra morphism  
  \begin{align}
   h_i: U_q(\ATT) \rightarrow  U_q(\g). 
  \end{align}
  \qed
\end{Proposition}
See \cite[Proposition 3.14]{BN} for the precise definitions of these maps. We only need the following facts which follow directly from definitions (see \cite[Proposition and Corollary 3.8]{Beck1}).
\begin{Lemma}
  The maps $h_i$ respect the triangular decomposition and the integral structure. Moreover,
  \begin{align}
    \label{eq:3}
    h_i \circ T_1 = T_i \circ h_i.
  \end{align}
\end{Lemma}
\begin{Lemma}
Under the map $h_i$,   
\begin{align}
  h_i\left( E^{\bar e}_{k\delta \pm r_i\alpha_1} \right) =  E^{\bar e}_{k d_i\delta \pm r_i\alpha_i}, \\
  h_i\left( E^{\bar s_1}_{k\delta \pm r_i\alpha_1} \right) =  E^{\bar s_i}_{k d_i\delta \pm r_i\alpha_i}.
\end{align}
\end{Lemma}

\begin{Theorem}{\label{thm:saito-imag}}
Fix $\bar w \in \bar W$, $j \in I$, and $i\in \bar I$. If $\alpha_j \prec_{\bar w} \delta$, then $\IPsi{\bar{s_j} \bar{w}}{ i,k}= T_{j} \IPsi{\bar{w}}{ i,k}$, and if $\alpha_j \succ_{\bar w} \delta$, then $\IPsi{\bar{s_j} \bar{w}}{ i,k}= T_{j}^{-1}\IPsi{\bar{w}}{ i,k}$. In particular, $\IPsi{\bar{w}}{ i,k} = T_w \IPsi{\bar{e}}{ i,k}$ where $w$ is the minimal length lift of $\bar{w}$ to $W$.

\end{Theorem}

\begin{proof} 
Proceed by induction on $\ell(\bar w)$. For the base case $\bar w=\bar e$ there are two possibilities:
\begin{itemize}

\item  If $i \neq j$, this follows because the real root vectors defining $\IPsi{\bar s_j}{ i ,k}$ are obtained exactly by applying $T_j$ to the real root vectors defining $\IPsi{\bar e}{i,k}$. 

\item If $i=j$, then  both $\IPsi{\bar e}{i,k}$ and $ \IPsi{\bar s_j}{ i ,k}$ lie in the image of $h_j$. Thus using Lemma \ref{thm:saito-imag}, we are reduced to the corresponding fact in affine rank-2 (either $\asl_2$ or $\ATT$) where it is known (see \cite[Lemma 4.4]{MT}). 
\end{itemize} 

So assume $\ell(\bar w) \geq 1$. We consider the case when $\alpha_j \prec_{\bar w} \delta$; the other case is similar. Let $\prec$ be an order of coarse type $\bar{w}$ whose least root is $\alpha_j$. Then for all positive real roots $\beta$ not equal to $\alpha_j$, we have $T_j (E^\prec_\beta) = E^{\prec^{s_j}}_{s_j(\beta)}$. We can use this, except in the case $\bar{w} \alpha_i = \alpha_j$, to obtain the result directly from the definition of $\IPsi{\bar{s_i} \bar{w}}{ i,k}$. This in particular covers the case when $j=0$. 

Also, since we can write $w = s_h u$ for some $h \in  \bar I$ and $u \in \overline{W}$ with $\ell(u) = \ell(w) - 1$, we can assume by induction that $\IPsi{\bar{w}}{ i,k} =T_w \IPsi{\bar{e}}{ i,k} $. Here $T_w$ corresponds to the minimal-length lift of $w$ to the affine Weyl group. 

Now consider the remaining case when $\bar{w} \alpha_i = \alpha_j$. We have $\ell(\bar w \bar s_i) = \ell(\bar w) +1$, $\ell(\bar s_j \bar w) = \ell(\bar w)+1$. Let $v = \bar s_j \bar w \bar s_i^{-1}$, then we also have $\ell(v \bar s_i) = \ell(v)+1$. So,
\begin{align}
T_j \IPsi{\bar{w}}{ i,k}  = T_jT_w \IPsi{\bar{e}}{ i,k} = T_v T_i \IPsi{\bar{e}}{ i,k} = T_v  \IPsi{\bar{s_i}}{ i ,k},
\end{align}
where the last equality is from the base case. 
Noting that $T_v E^{\bar{s_i}}_{\alpha_i + n\delta} = E^{\bar{v}\bar{s_i}}_{v(\alpha_i) + n\delta}$ for $n \geq 0$, and $T_v E^{\bar{s_i}}_{-\alpha_i + n\delta} = E^{\bar{v}\bar{s_i}}_{-v(\alpha_i) + n\delta}$ for $n >0$, we conclude that
\begin{align} 
T_v  \IPsi{\bar{s_i}}{ i ,k} = \IPsi{\bar{v} \bar{s_i}}{i,k} = \IPsi{\bar{s_j} \bar{w}}{i,k}.
\end{align}

\end{proof}

\begin{Corollary} \label{cor:im-sym}
Fix $\bar w \in \bar W$ and $j \in I$, and let $i \in \bar I$ be arbitrary. 
If $\alpha_j \prec_{\bar w} \delta$, then $S_{i}^{\bar{s_j} \bar{w},\lambda}= T_{j} S_{ i}^{\bar w,\lambda}.$
If $\alpha_j \succ_{\bar w} \delta$, then $S_{i}^{\bar{s_j} \bar{w},\lambda}= T_{j}^{-1} S_{ i}^{\bar w,\lambda}.$
\qed
\end{Corollary}

\begin{Proposition} \label{prop:alternative-psi-formula}
Let  $\bar w$ be a coarse type, let $i \in \bar I$, and let $\gamma_i$ be as in Definition \ref{def:im-root-vectors}.  Let $\ell>0, m \geq 0$ be integers. Let $k = \ell + m$. Then
\begin{align}
  \IPsi{\bar w}{ i,k}= E^{\bar{w}}_{\ell d_i\delta-\gamma_i}E^{\bar{w}}_{m d_i\delta +{\gamma_i}}  - q_i^{-2} E^{\bar{w}}_{m d_i\delta +{\gamma_i}}E^{\bar{w}}_{\ell d_i\delta-\gamma_i}.
\end{align}
\end{Proposition}

\begin{proof}
The case $\bar w=\bar e$ holds by \cite[Proposition 1.2]{BCP} and \cite[Proposition 3.26 (2)]{Aka}. The general case follows by Theorem \ref{thm:saito-imag}.
\end{proof}

\begin{Definition}
Let $*$ be the involution on $\bar I$ defined by $\alpha_{i*} = -{\bar w_0} (\alpha_i)$. 
\end{Definition}

\begin{Proposition}{\label{prop:star-imag}}
Let  $\bar w$ be a coarse type, let $i \in \bar I$, and let $\lambda$ be a partition. Then
\begin{align}
(S^{\bar w,  \lambda}_i)^*= S^{\bar w_0 \bar w,  \lambda}_{i^*}
\end{align}
\end{Proposition}

\begin{proof}
It suffices to show the same thing for the $\IPsi{\bar w}{ i,k}$.  Let $\gamma_i$ be as in Definition \ref{def:im-root-vectors}. Then  
\begin{align}
(\IPsi{\bar w}{ i,k})^* & = (E^{\prec}_{k d_i\delta - \gamma_i} E^{\prec}_{\gamma_i} - q_i^{-2}E^{\prec}_{\gamma_i} E^{\prec}_{k d_i\delta-\gamma_i } )^* \\
\label{eq:dd} & = (E^{\prec}_{\delta - \gamma_i} E^{\prec}_{(k-1) d_i\delta+\gamma_i} - q_i^{-2}E^{\prec}_{(k-1) d_i\delta+\gamma_i} E^{\prec}_{\delta-\gamma_i } )^* \\
& \label{eq:*in} = E^{\prec^*}_{k d_i\delta-(\delta- \gamma_i)} E^{\prec^*}_{\delta - \gamma_i} - q_i^{-2}E^{\prec^*}_{\delta-\gamma_i } E^{\prec^*}_{k d_i\delta-(\delta- \gamma_i)} = \IPsi{\bar w_0 \bar w}{ i^*,k}.
\end{align}
Here \eqref{eq:dd} follows from Proposition \ref{prop:alternative-psi-formula} and \eqref{eq:*in} follows from Lemma \ref{lem:real*}.
\end{proof}

\subsection{PBW bases}

Let $\prec$ be a convex order and $\cc$ a Lusztig datum. Define the corresponding PBW basis element by
\begin{align}
L(\cc,\prec) = E_{\beta_1}^{\prec,(\cc_{\beta_1})} \cdots  E_{\beta_N}^{\prec,(\cc_{\beta_N})}  S^{\prec,\cc_\delta} E_{\gamma_M}^{\prec,(\cc_{\gamma_M})} \cdots  E_{\gamma_1}^{\prec,(\cc_{\gamma_1})},
\end{align}
where $\beta_1 \prec \cdots \prec \beta_N \prec \delta \prec \gamma_M \prec \cdots \prec \gamma_1$ are the real roots for which $c_{\beta} \neq 0$.

\begin{Proposition} \label{prop:cCS}

Let $\prec$ be a convex order, and let $i_\text{left}, i_\text{right} \in I$ be the vertices in Dynkin diagram corresponding to the least and greatest roots of the order $\prec$ respectively. 

\begin{enumerate}
\item \label{el1} Let $\cc^\prime$ be the Lusztig datum such that $\cc^\prime_{\alpha_{i_\text{left}}}= \cc_{\alpha_{i_\text{left}}}+1$ and otherwise agrees with $\cc$. Then $\widetilde{ E}_{i_\text{left}} L(\cc,\prec) = L(\cc^\prime,\prec)$.
\item \label{el2} Let $\cc^\prime$ be the Lusztig datum such that $\cc^\prime_{\alpha_{i_\text{right}}}= \cc_{\alpha_{i_\text{right}}}+1$ and otherwise agrees with $\cc$. Then $\widetilde{ E}^*_{i_\text{right}} L(\cc,\prec) = L(\cc^\prime,\prec)$.
\item \label{el3} Suppose  $\cc_{\alpha_{i_\text{left}}}=0$. Then $T_{i_\text{left}} L(\cc,\prec) =L(\cc \circ s_i, \prec^{s_i})$
\item \label{el4} Suppose  $\cc_{\alpha_{i_\text{right}}}=0$. Then $T^{-1}_{i_\text{right}} L(\cc,\prec) = L(\cc \circ s_i, \prec^{s_i})$
\item \label{el5} $L(\cc,\prec)^* = L(\cc,\prec^*)$.
\end{enumerate}
Here $\cc \circ s_i$ is the Lusztig datum defined by $(\cc \circ s_i)_\beta = s_i(\beta)$ for $\beta \neq \alpha_i$ and $(\cc \circ s_i)_{\alpha_i} = 0$.
\end{Proposition}

\begin{proof}
Fix $N >0$. We can find a one-row order $\prec^\prime$ such that $L(\cc,\prec)= L(\cc,\prec^\prime)$ for all $\cc$ such that $\wt (\cc)$ has height below $N$. Since for fixed $\cc$ each statement involves only at most two weight spaces of $U^+$, without loss of generality we may assume that $\prec$ is a one-row order.

Statement \eqref{el1} follows by the definition of $\widetilde E_i$. Statement \eqref{el3} follows from the fact that $T_{i_\text{left}} S^{\prec,\cc_\delta}  = S^{\prec^{s_{i_\text{left}}}, \cc_\delta}$. The part of statement \eqref{el5}  involving real root vectors follows from the fact that $* \circ T_i \circ * = T^{-1}_i$ for all $i$; the part involving imaginary root vectors is Proposition \ref{prop:star-imag}. Putting these facts together using the fact that $*$ is an anti-involution, we get statement \eqref{el5}. Statements \eqref{el2} and \eqref{el4} follow from \eqref{el1} and \eqref{el3} respectively after applying \eqref{el5}.
\end{proof}

\subsection{Comparison with Beck and Nakajima's PBW bases}
We now recall Beck and Nakajima's construction from  \cite{BN}. For each $i \in \bar I$, choose a diagram automorphism $\tau_i \in \mathcal{T}$ such that $t_{\omega_i} \tau_i^{-1} \in W$ (the non-extended affine Weyl group). Choose a reduced decomposition for each $t_{\omega_i} \tau_i^{-1}$. Concatenating these expressions and commuting the $\tau_i$-factors to the right gives a reduced expression
\begin{align}
  \label{eq:5}
t_{\omega_n} t_{\omega_{n-1}} \cdots t_{\omega_1} = s_{i_1} \cdots s_{i_N} \tau
\end{align}
where $\tau = \tau_n \cdots \tau_1$. Form the doubly-infinite word 
\begin{align}
  \label{eq:6}
  {\mathbf h} = ( \cdots i_{-1}, i_0, i_1, \cdots )
\end{align}
defined by $i_{k+N} = \tau(i_k)$ for all $k \in \ZZ$.
Following \cite[Equation 3.3]{BN}, define
\begin{align}
  \label{eq:7}
\beta_k =  
\begin{cases}
s_{i_0} s_{i_{-1}} \cdots s_{i_{k+1}} \alpha_{i_k} & \quad \text{if } k \leq 0, \\
s_{i_1} s_{i_2} \cdots s_{i_{k-1}} \alpha_{i_k} & \quad \text{if } k > 0.
\end{cases}
\end{align}
This defines a convex order $\prec_0$ on positive roots by
\begin{align}
  \label{eq:8}
  \beta_0 \prec_0 \beta_{-1} \cdots \prec_0 \delta \prec_0 \cdots \beta_2 \prec_0 \beta_1.
\end{align}
For each $p \in \ZZ$, define a convex order as follows:
\begin{align}
  \label{eq:9}
\prec_p=
\begin{cases}
\prec_0^{{s_{i_0}} \cdots s_{i_{n-1}}} & \quad \text{if } n\geq 0 \\
\prec_0^{{s_{i_{-1}}} \cdots s_{i_{n}}} & \quad \text{if } n< 0.
\end{cases}
\end{align}

The PBW basis $L(\cdot, 0)$ considered by Beck and Nakajima in \cite[Section 3.1]{BN} is precisely our $L(\cdot, \prec_0)$. They also consider a basis $L(\cdot, p)$ for each $p \in {\Bbb Z}$, which are constructed from $L(\cdot, \prec_0)$ using braid group reflections. The following is immediate from their construction and our Proposition \ref{prop:cCS}.
\begin{Corollary}
For all $p$, Beck and Nakajima's basis $  L(\cdot, p)$ is precisely the same as our basis $L(\cdot, \prec_p)$. In particular, Beck and Nakajima's bases $L(\cdot, p)$ and $L(\cdot, q)$ coincide whenever $p-q$ is divisible by $N$ times the order of $\tau$. \qed
\end{Corollary} 

\begin{Remark} With Beck and Nakajima's convex order, there are good formulas relating the real root vectors to the generators in Drinfeld's presentation of $U_q(\mathfrak{g})$ (see \cite[Equation 3]{Beck2} and \cite[Lemma 1.5]{BCP}). 
\end{Remark}

\subsection{More properties of PBW bases}

\begin{Proposition}
  \label{prop-crystal-basis}
For all convex orders, $L(\cc,\prec) \in \Uplusintegral$, and $\{L(\cdot,\prec) \}$ is a crystal basis. 
\end{Proposition}

\begin{proof}  
For Beck and Nakajima's bases this is \cite[Theorem 3.13]{BN}. The fact that the $L(\cc,\prec)$ lie in $\Uplus_{\mathscr{A}}$ reduces to the fact that $S^{\prec_0,\overline \lambda} \in\Uplus_{\mathscr{A}}$ when $\prec_0$ is of the standard coarse type (see \cite[Proposition 3.15]{BN}). We can also compute the inner products $\left(L(\cc,\prec),L(\cc',\prec) \right) = \delta_{\cc,\cc'} \text{ mod } {q_s}^{-1}$ exactly as in \cite[Equation 3.25]{BN} i.e the basis $\{ L(\cc,\prec)$ is ``almost orthonormal''. Then by \cite[14.2.2]{Lus} there exists a sign $\sgn(\cc,\prec) \in \{ \pm 1 \}$ such that  $\sgn(\cc,\prec) L(\cc,\prec) \in B(-\infty)$. Since both Kashiwara operators and braid operators preserve that sign, we are reduced to checking the sign for $S^{\prec_0,\underline \lambda}$. The fact that this sign is $+1$ is \cite[Lemma 5.2]{BN}. 
\end{proof}

\begin{Definition}
 Let $w \in W$. Define
\begin{align}
  \label{eq:52}
U^+_q(w,+) = \{x \in \Uplus \mid T_w(x) \in \Uplus\}, \;\; \text{ and} \\
U^+_q(w,-) =  \{x \in \Uplus \mid T_w(x) \in U_q^-(\g)\}. 
\end{align}
The vector spaces $U^+_q(w,\pm)$ are subalgebras of $\Uplus$.
\end{Definition}

Let $\Delta^-_w$ be the set of positive roots that are made negative by $w$, and let $\Delta^+_w$ be the set of positive roots that remain positive after being acted upon by $w$. Then $(\Delta^-_w, \Delta^+_w)$ is a biconvex partition of the positive roots. 

\begin{Proposition} \label{prop:lop}
Fix $w \in W$ and let $\prec$ be a convex order such that for all $\beta^- \in \Delta^-_w$ and $\beta^+ \in \Delta^+_w$ we have $\beta^- \prec \beta^+$. The PBW basis vectors that only involve root vectors corresponding to roots in $\Delta^\pm_w$ form a basis of $U^+_q(w,\pm)$. In particular, their span is independent of choice of $\prec$. 
\end{Proposition}
\begin{proof}
It is clear from their definitions that such basis vectors are elements of $U^+_q(w,\pm)$, and that they are linearly independent. Moreover the number of such vectors of a fixed weight $\lambda$ is precisely $\dim U^+_q(w,\pm)$ (i.e. it is given by the number of ways of writing $\lambda$ as a sum of elements of $\Delta_w^\pm$ counted with multiplicity). 
\end{proof}

For any convex order $\prec$ and any prefix $S$ of that order, let $\bar S$ be the complement of $S$ in $\Delta_+^{\text{min}}$. We can factor any PBW basis element $L(\cc,\prec) = L(\cc_S, \prec) \cdot L(\cc_{\bar S}, \prec)$, where $\cc_S$ is the Lusztig datum that agrees with $\cc$ for all roots in $S$ and is otherwise zero (similarly for $\cc_{\bar S}$). The following is a slight strengthening of \cite[Lemma 3.30]{BN}.

\begin{Lemma}{\label{lem:upper-triangularity-of-multiplication}}
Let $\prec$ be a convex order, and let $\cc$ and $\cc^\prime$ be Lusztig data. Write
\begin{align}
L(\cc,\prec) L(\cc^\prime, \prec) = \sum_{\cc^{\prime\prime}} a^{\cc^{\prime\prime}}_{\cc,\cc^\prime} L(\cc^{\prime\prime},\prec)
\end{align}
where $ a^{\cc^{\prime\prime}}_{\cc,\cc^\prime} \in \QQ(q_s)$. If $ a^{\cc^{\prime\prime}}_{\cc,\cc^\prime} \neq 0$ then, for any prefix $S$ of $\prec$, we have:
\begin{align}
  \label{eq:38}
  \wt (\cc''_{S}) \geq \wt (\cc_{S}) \\ 
  \wt (\cc''_{\bar S}) \geq \wt (\cc'_{\bar S}) 
\end{align}
\end{Lemma}

\begin{proof}
Let $S$ be a prefix of $\prec$. Then
\begin{align}
L(\cc,\prec)L(\cc',\prec)  & = L(\cc_S,\prec)L(\cc_{\bar S},\prec)L(\cc'_{S},\prec)L(\cc'_{\bar S},\prec) 
\end{align}
Rewriting and factoring according to $S$, we have:
\begin{align}
  \label{eq:23}
 L(\cc_{\bar S},\prec)L(\cc'_{S},\prec) = \sum b_{\cc,\cc',S}^{\tilde{\cc}} L(\tilde{\cc}_S,\prec)L(\tilde{\cc}_{\bar S},\prec) 
\end{align}
Combining this with the above equation, we have:
\begin{align}
  \label{eq:24}
L(\cc,\prec)L(\cc',\prec)  & = \sum b_{\cc,\cc',S}^{\tilde{\cc}} L(\cc_S,\prec)L(\tilde{\cc_{ S}},\prec)L(\tilde{\cc}_{\bar S},\prec)L(\cc'_{\bar S},\prec) 
\end{align}
Then by Proposition \ref{prop:lop}, when we write
\begin{align}
  \label{eq:25}
 L(\cc,\prec)L(\cc',\prec)  & =  \sum_{\cc^{\prime\prime}} a^{\cc^{\prime\prime}}_{\cc,\cc^\prime} L(\cc^{\prime\prime},\prec)
\end{align}
for every $\cc''$ that appears, there is some $\tilde{\cc}$ such that: 
\begin{align}
  \label{eq:26}
  \wt (\cc''_{S}) = \wt (\cc_{S}) + \wt (\tilde{\cc}_{S}) \\ 
  \wt (\cc''_{\bar S}) = \wt (\cc'_{\bar S}) + \wt (\tilde{\cc}_{\bar S}). 
\end{align}
\end{proof}

We recall the following definition from \cite{BN}.
\begin{Definition}
Fix a convex order $\prec$. Let $S$ be the set of real roots less than $\delta$ according to $\prec$. Let $\cc$ and $\cc'$ be Lusztig data of the same weight. Then we say
\begin{align}
  \label{eq:18}
  \cc \geq_\ell \cc'
\end{align}
if $\cc_S$ is greater than or equal to $\cc'_S$ in the lexicographical order where we order the roots of $S$ in the same order as $\prec$. Explicitly, $\cc \geq_\ell \cc'$, means that either $\cc_S = \cc'_S$, or there is some real root $\beta \in S$, such that $\cc_\gamma = \cc'_\gamma$ for all $\gamma \prec \beta$, and $\cc_\beta > \cc'_\beta$. 

Similarly, let $S'$ be the set of real roots greater than $\delta$. We define 
\begin{align}
  \label{eq:19}
 \cc \geq_r \cc' 
\end{align}
if $\cc_{S'}$ is greater than or equal to $\cc'_{S'}$ in the lexicographical order where we order the roots of $S'$ in the \emph{reverse} order as $\prec$. 
Explicitly, $\cc \geq_r \cc'$, means that either $\cc_{S'} = \cc'_{S'}$, or there is some real root $\beta \in S'$, such that $\cc_\gamma = \cc'_\gamma$ for all $\gamma \succ \beta$, and $\cc_\beta > \cc'_\beta$.

Define 
\begin{align}
  \label{eq:34}
  \cc > \cc'
\end{align}
if $\cc \geq_\ell \cc'$ and $ \cc \geq_r \cc'$, and one of those inequalities is strict. Note that $>$ is a partial order, but $\geq_\ell$ and $\geq_r$ are only preorders.
\end{Definition}

The following is immediate from Lemma \ref{lem:upper-triangularity-of-multiplication} (see also {\cite[Lemma 3.30]{BN}).

\begin{Corollary}
  \label{cor:beck-nakajima-upper-triangularity}
Let $\prec$ be a convex order, and let $\cc$ and $\cc^\prime$ be Lusztig data. Write
\begin{align}
L(\cc,\prec) L(\cc^\prime, \prec) = \sum_{\cc^{\prime\prime}} a^{\cc^{\prime\prime}}_{\cc,\cc^\prime} L(\cc^{\prime\prime},\prec)
\end{align}
where $ a^{\cc^{\prime\prime}}_{\cc,\cc^\prime} \in \QQ(q_s)$. Then
\begin{align}
  \label{eq:40}
 \cc'' \geq_\ell \cc \quad \text{ and }  \quad
 \cc'' \geq_r \cc'.
\end{align}
\end{Corollary}

The following is a slight generalization of \cite[Proposition 3.36]{BN}. Their proof carries over to one-row convex orders, and the case of general convex orders follows by approximation by one-row orders. 

\begin{Proposition}  \label{prop:bbar}
For any convex order $\prec$ and any $\cc$,
\begin{align}
\overline{L(\cc,\prec)} = L(\cc,\prec) + \sum_{\cc^\prime > \cc} a_{\cc,\cc^\prime} L(\cc^\prime, \prec)
\end{align}
for $a_{\cc,\cc^\prime} \in \mathbb{Q}(q_s)$.  \qed
\end{Proposition}

\begin{Theorem} \label{th:UpperTriangularity}
For any convex order $\prec$, 
the change of basis from $L(\cdot, \prec)$ to the canonical basis $\B$ is unit upper triangular with respect to $>$. That is,
$$b(\cc, \prec)=L(\cc, \prec) + \sum_{\cc'>\cc} a_{\cc,\cc'} L(\cc',\prec).$$  
\end{Theorem}

\begin{proof}
By e.g. \cite[\S 5.1]{Lec} (see also \cite[Theorem 5.3]{LusNotes}), Proposition \ref{prop:bbar} implies that there is a unique basis $\B'$ such that
\begin{enumerate}

\item $\text{span}_{{\Bbb Z}[q_s^{-1}]} \B' = \text{span}_{{\Bbb Z}[q_s^{-1}]}  L(\cdot, \prec)$. 

\item $\B' \equiv L(\cdot, \prec)$ mod $q_s^{-1}$

\item $\B'$ is bar-invariant. 
\end{enumerate}
and furthermore the change of basis from $L(\cdot, \prec)$ to $\B'$ is unit triangular. By Theorem \ref{th:defofB} $\B$ satisfies all three of these conditions  so $\B=\B'$.
\end{proof}

\begin{Proposition}{\label{CanonicalBasisCrystalFormula-prime}}
Fix $b \in \cB$. Write $b E^{(n)}_i  = \sum_{b^\prime} a_{b,b^\prime} b^\prime$. Then $a_{b,(\e_i^*)^n b} \neq 0$. Similarly, if we write $ E^{(n)}_i b = \sum_{b^\prime} \tilde{a}_{b,b^\prime} b^\prime$. Then $\tilde{a}_{b,(\e_i)^n b} \neq 0$.
\end{Proposition}

\begin{proof}
Fix a convex order $\prec$ with $\alpha_i$ maximal. 
There is a unique Lusztig datum $\cc$ such that $b = b(\cc,\prec)$.  By Theorem \ref{th:UpperTriangularity},
\begin{align}
b(\cc,\prec) = L(\cc,\prec) + \sum_{\cc^\prime > \cc} a_{\cc,\cc^\prime} L(\cc^\prime,\prec).
\end{align}
Multiplying both sides by $E_i^{(n)}$ on the right, 
\begin{align}
b(\cc,\prec) E^{(n)}_i = L(\cc,\prec) E^{(n)}_i + \sum_{\cc^\prime > \cc} a_{\cc,\cc^\prime} L(\cc^\prime,\prec) E^{(n)}_i.
\end{align}
Since $\alpha_i$ is maximal, each $L(\dd,\prec) E^{(n)}_i$ is a multiple of  $(\e_i^*)^n L(\dd,\prec)$. Now rewrite each term using the canonical basis. Again using Theorem \ref{th:UpperTriangularity}, only $L(\cc,0)E^{(n)}_1$ will contribute to the coefficient of $ (\e_i^*)^n b(\cc,\prec)$, and its contribution is clearly non-zero. 

The other statement follows similarly.
\end{proof}

\section{PBW polytopes}

By Proposition \ref{prop-crystal-basis}, for each pair of convex orders $\prec$, $\prec'$ there is a bijection $\cc \leftrightarrow \cc'$ on Lusztig data such that
\begin{align}
 \label{eq:10}
L(\cc, \prec) \equiv  L(\cc', \prec) \mod q_s^{-1}.
\end{align}
The collection of all the Lusztig data (for all convex orders) corresponding to an element $b \in B(-\infty)$ fit together to form a decorated polytope. We study these polytopes, and show that they agree with previous definitions of affine MV polytope. 

For orders of the form $\prec_p$, Beck and Nakajima pose the question \cite[Remark 3.29]{BN} of describing this bijections \eqref{eq:10} combinatorially. Our construction shows that the answer is precisely recorded by affine MV polytopes.

\subsection{Construction}
Fix $b \in B(-\infty)$ and  $w \in W$. Recall that $\Delta^-_w$ denotes the set of positive roots that are inverted by $w$ and that $\Delta^+_w$ denotes the set of positive roots that are remain positive under $w$. Let $\prec$ be a convex order such that all elements of $\Delta^-_w$ are less than all elements of $\Delta^+_w$. Then there is a PBW basis vector $L(\cc,\prec)$ such that $L(\cc,\prec)+q_s^{-1} \mathcal{L}=b$. Moreover, because $\prec$ is compatible with the biconvex partition $(\Delta^-_w, \Delta^+_w)$, we see that $L(\cc,\prec) = L_1 L_2$ where $L_1$ is a monomial in root vectors for roots in $\Delta^-_w$, and $L_2$ is a monomial in root vectors for roots in $\Delta^+_w$. Set $\mu^+_w(b) = \text{wt}(L_1)$. Similarly, define 
$\mu^-_w(b)$ using the biconvex partition $(\Delta^+_w, \Delta^-_w)$.

\begin{Proposition} $\mu^+_w(b)$ (resp. $\mu^-_w(b)$) depend only on $(\Delta^-_w, \Delta^+_w)$ (resp. $(\Delta^+_w, \Delta^-_w)$), not additionally on $\prec$.
\end{Proposition}

\begin{proof} 
It suffices to check this for $\mu^+_w(b)$. Let $\prec'$ be another convex order compatible with  $(\Delta^-_w, \Delta^+_w)$, and expand $L_1$ in the $\prec'$ PBW basis. By Proposition \ref{prop:lop} every basis element that appears involves only roots in $\Delta^-_w$, so it expands as $L_1'+ q_s^{-1} \sum p_j (q_s^{-1}) L_1^{(j)}$ for some $\prec'$ PBW basis element $L_1'$. Similarly, if we can expand $L_2$ as as $L_2'+ q_s^{-1} \sum p_j (q_s^{-1}) L_2^{(j)}$.
Multiplying these manifestly gives a sum of $\prec'$ PBW basis elements, and modulo $q_s^{-1}$ we get $L_1'L_2'$. Thus this is the factorization of $b$ using the PBW basis $\prec'$. Certainly $\wt(L'_1)= \wt(L_1)$.
\end{proof}

\begin{Lemma}{\label{lem:cone-condition}}
  Let $b \in B(-\infty)$, and let $\Delta_+= S_1 \sqcup S_2 = S'_1 \sqcup S'_2$ be two biconvex partitions. Then $\mu_{(S'_1,S'_2)}(b) - \mu_{(S_1,S_2)}(b)  \in \text{span}_{{\Bbb Z}_{\geq 0}} ( S_1 \cup - S_2 ).$ 
\end{Lemma}

\begin{proof}
Let $\prec$ be a convex order such that $S_1 \prec S_2$. Similarly define $\prec'$. Let $\cc$ and $\cc'$ be the Lusztig data such that
\begin{align}
  \label{eq:28}
 b \equiv L(\cc,\prec) \equiv L(\cc',\prec').
\end{align}
We can factor
\begin{align}
  \label{eq:29}
 L(\cc,\prec) = L(\cc_{S_1},\prec)L(\cc_{S_2},\prec),  \\
 L(\cc',\prec) = L(\cc'_{S'_1},\prec)L(\cc'_{S'_2},\prec).  
\end{align}
Expanding in the $\prec$-basis, we have
\begin{align}
  \label{eq:30}
 L(\cc'_{S'_1},\prec') = \sum_{\cc^{(1)}} a_{\cc^{(1)}}L(\cc^{(1)}_{S_1},\prec) L(\cc^{(1)}_{S_2},\prec), \\ 
 L(\cc'_{S'_2},\prec') = \sum_{\cc^{(2)}} a_{\cc^{(2)}}L(\cc^{(2)}_{S_1},\prec) L(\cc^{(2)}_{S_2},\prec).
\end{align}
So, 
\begin{align}
  \label{eq:030}
 L(\cc'_{S'_1},\prec')L(\cc'_{S'_2},\prec') = \sum_{\cc^{(1)},\cc^{(2)}} a_{\cc^{(1)}}a_{\cc^{(2)}}L(\cc^{(1)}_{S_1},\prec) L(\cc^{(1)}_{S_2},\prec)  L(\cc^{(2)}_{S_1},\prec) L(\cc^{(2)}_{S_2},\prec).
\end{align}
We can expand
\begin{align}
  \label{eq:31}
L(\cc^{(1)}_{S_2},\prec)  L(\cc^{(2)}_{S_1},\prec) = \sum_{\cc''} b_{{\cc^{(1)}},{\cc^{(2)}}} ^{\cc''} L(\cc''_{S_1},\prec)L(\cc''_{S_2},\prec).
\end{align}
Combining this, 
\begin{align}
  \label{eq:32}
 L(\cc'_{S'_1},\prec')L(\cc'_{S'_2},\prec')   = \hspace{-0.5cm}  \sum_{\cc^{(1)},\cc^{(2)},\cc''}  \hspace{-0.4cm}  
  a_{\cc^{(1)}}a_{\cc^{(2)}}b_{{\cc^{(1)}},{\cc^{(2)}}} ^{\cc''} L(\cc^{(1)}_{S_1},\prec) L(\cc''_{S_1},\prec)L(\cc''_{S_2},\prec)L(\cc^{(2)}_{S_2},\prec).
\end{align}
When expanding in the $\prec$-basis $L(\cc',\prec)$ will appear, and, by Proposition \ref{prop:lop}, 
\begin{align}
  \label{eq:33}
 \wt( \cc_{S_1}) = \wt( \cc^{(1)}_{S_1}) + \wt( \cc''_{S_1}) 
\end{align}
for some choices of $\cc^{(1)}$,$\cc^{(2)}$, and $\cc''$.
By \eqref{eq:30},
\begin{align}
  \label{eq:35}
 \wt(\cc'_{S'_1}) = \wt(\cc^{(1)}_{S_1}) + \wt(\cc^{(1)}_{S_2}). 
\end{align}
Therefore
\begin{align}
  \label{eq:36}
  \mu_{(S'_1,S'_2)}(b) - \mu_{(S_1,S_2)}(b) = \wt(\cc'_{S'_1}) - \wt( \cc_{S_1}) =  \wt( \cc''_{S_1}) - \wt(\cc^{(1)}_{S_2}).  
\end{align}
\end{proof}

\begin{Definition}
Fix $b \in B(-\infty)$. The {\bf undecorated PBW polytope} 
$\overset{\circ}{\text{PBW}}_b$ is the convex hull of $\{ \mu^\pm_w(b) \}$.
\end{Definition}

It is immediate from Lemma \ref{lem:cone-condition} each $\mu^\pm_w(b)$ is a vertex of $\overset{\circ}{\text{PBW}}_b$.

\begin{Theorem} \label{thm:ispw} $\overset{\circ}{\text{PBW}}_b$ is an undecorated affine pseudo-Weyl polytope whose edge lengths record the real Lusztig data of all the PBW basis vectors corresponding to $b$. 
\end{Theorem}
\begin{proof} 
It remains to check that every edge is parallel to a root. So suppose we have a non-degenerate edge that is not parallel to a root. That edge is the maximal set of some linear functional $\phi$ which does not vanish on any root. Let $S_1$ (resp. $S_2$ be the set of positive roots on which $\phi$ takes negative values (resp. positive values). Then $(S_1, S_2)$ is a biconvex partition of the positive roots. 

Consider the vertex $\mu_{(S_1,S_2)}(b)$. Any other vertex of $\text{PBW}_b$ is of the form $\mu_{(S'_1,S'_2)}(b)$ for some other biconvex partition $(S'_1,S'_2)$. By Lemma \ref{lem:cone-condition}:
\begin{align}
  \label{eq:37}
\mu_{(S'_1,S'_2)}(b) \subset \mu_{(S_1,S_2)}(b)+ \text{span}_{{\Bbb R}_+} S_1 \cup (-S_2)
\end{align}
Therefore, we have $\phi\left(\mu_{(S'_1,S'_2)}(b)\right) < \phi\left(\mu_{(S_1,S_2)}(b)\right)$ unless $\mu_{(S'_1,S'_2)}(b) = \mu_{(S_1,S_2)}(b)$. So the maximal set of $\phi$ is a point, which contradicts our assumption.
\end{proof}

In fact the edges of $\overset{\circ}{\text{PBW}}_b$ parallel to $\delta$ are naturally decorated:
 for each $\bar w \in \bar W$,
 choose a convex order $\prec$ of that coarse type, and consider the imaginary part of the PBW monomial corresponding to $b$. This is indexed by a family of partitions $\{\lambda_{\bar w \bar \alpha_i}\}_{i \in \bar I}$. 
It is clear from Definition \ref{def:im-root-vectors} that these partitions are independent of the choice of $\prec$.  
\begin{Definition}
The PBW polytope $\PBW_b$ of $b \in B(-\infty)$ is $\overset{\circ}{\text{PBW}}_b$ along with the decoration described above. 
\end{Definition}
One can easily see that the conditions from Definition \ref{def:edge-decorated} relating edge lengths with the sizes of the partitions $\lambda_{\bar w \bar \alpha_i}$ are satisfied, so ${\text{PBW}}_b$ is an edge-decorated affine pseudo-Weyl polytope.

\subsection{Proof that PBW polytopes are MV polytopes} 

\begin{Theorem} \label{thm:PBWMV}
For each $b \in B(-\infty)$, $PBW_b$ arises from a facet-decorated polytope, and this agrees with the MV polytope $MV_b$ constructed in \cite{TW}. 
\end{Theorem}
To prove \ref{thm:PBWMV} it suffices to show that the map $b \rightarrow PBW_b$ satisfies the conditions of our theorem characterizing affine MV polytopes (Theorem \ref{th:MV-def}). Conditions (C) and (S) are immediate from Proposition \ref{prop:cCS}.
It remains to establish (I). This requires a few preliminary lemmas.

\begin{Lemma} \label{lem:elex-prime} Let $i\in \bar I$, and let $\prec$ be a convex order for which $\alpha_i$ is minimal.
 Let $i \in \bar I$, and suppose that $\prec$ is an order such that $\alpha_i \prec \delta$. 
Then, for all $b \in B(-\infty)$ we have $\cc^\prec(b) \leq_\ell\cc^\prec(\tilde e^*_i(b))$ and $\cc^{\prec^{s_i}}(b) \leq_r \cc^{\prec^{s_i}}(\tilde e_i(b))$. 
\end{Lemma}
\begin{proof}
Identify $b$ with the corresponding canonical basis element. By Theorem \ref{th:UpperTriangularity}, 
$$b= L(\cc^\prec(b), \prec)+ \text{ $\geq_\ell$ greater PBW  terms,}$$
so 
$$
\begin{aligned}
b E_i= L(\cc^\prec(b), \prec) E_i+  (\text{  $\geq_\ell$ greater PBW terms })E_i.
\end{aligned}$$
Now re-expand this in the PBW basis. By Corollary \ref{cor:beck-nakajima-upper-triangularity}, all Lusztig data that appear are $ \geq_\ell \cc$. 
Using Theorem \ref{th:UpperTriangularity} again, when $b E_i $ is expanded in the canonical basis, all Lusztig data that appear are still $ \geq_\ell \cc$. By Proposition \ref{CanonicalBasisCrystalFormula-prime} $L(\cc^\prec(\tilde e^*_i(b), \prec)$ must appear in this sum, so $\cc^\prec(b) \leq_\ell\cc^\prec(\tilde e^*_i(b))$.

By a similar argument, we also have the other inequality.
\end{proof}

\begin{Lemma} \label{lem:trap}
Let $i\in \bar I$, and let $\prec$ be a convex order for which $\alpha_i$ is minimal. Let $\underline \lambda$ be a multipartition, and suppose that $b \in B(-\infty)$ is such that $\cc^\prec(b)= \underline \lambda$.  Then 
$\cc_{\alpha_i}^{\preci}(b)= r_i \cdot \cc_{d_i\delta-r_i\alpha_i}^{\preci}(b)$, and $\cc^{\preci}_\beta(b)=0$ for all other real roots $\beta$. 

We also have $\cc_{\alpha_i}^{\prec}(\sigma_ib)= r_i \cdot \cc_{d_i\delta-r_i\alpha_i}^{\prec}(\sigma_ib)$ and $\cc^{\prec}_\beta(\sigma_ib)=0$ for all other real roots $\beta$.
\end{Lemma}
\begin{proof}
Assume to the contrary that $\cc_\beta^{\preci}(b) \neq 0$ for some real $\beta$ other then $\alpha_i, d_i \delta-r_i\alpha_i$. 
Without loss of generality may assume $\beta \succ^{s_i} \delta$. 

Since $\cc^\prec(b) = \underline \lambda$ we must have $\cc^\preci(\sigma_i(b)) = \underline \lambda$. But
\begin{align}
  \label{eq:4}
\sigma_i(b)= {\tilde e}_i^{\varphi_i^*(b)} ({\tilde f}_i^*)^{\varphi_i^*(b)} b.
\end{align}
Certainly $\cc_{\beta}^{\preci}( (f_i^*)^{\varphi_i^*(b)} b) \neq 0$, as all Lusztig data with respect to $\preci$ agree with those for $b$ except $\alpha_i$. Thus by 
Lemma \ref{lem:elex-prime}, $ \cc_{\beta'}^{\preci}(e_i^{\varphi_i^*(b)} (f_i^*)^{\varphi_i^*(b)} b) \neq 0$ as well for some $\delta \preci \beta \preceq^{s_i} \beta' $. But this contradicts  $\cc^\preci(\sigma_i(b)) = \underline \lambda$.

The statements for $\sigma_ib$ follow by a similar argument.
\end{proof}

Fix $\prec$ with $\alpha_i$ minimal. Fix a multipartition $\underline \lambda$, and choose $b \in B(-\infty)$ with $\cc_\delta^\prec(b)=\underline \lambda$ and $\cc_\beta^\prec(b) = 0$ for all real roots $\beta$. By Lemma \ref{lem:trap} we have $\cc_{\alpha_i}^{\preci}(b)= r_i \cdot \cc_{d_i\delta-r_i\alpha_i}^{\preci}(b)=r_i \cdot a$ for some non-negative integer $a$, and $\cc^{\preci}_\beta(b)=0$ for all other real roots $\beta$. 

\begin{Lemma}{\label{lem:equality-of-phis}}
 In the above situation, we have  $\varphi_i(\sigma_ib) = \varphi_i^*(b) = r_i \cdot a$.
\end{Lemma}
\begin{proof}
For the second equality, we compute $\varphi_i^*(b)= \cc^{\preci}_{\alpha_i}(b) = r_i a$.
 We know that $\sigma_i b = {\tilde e}_i^{\varphi_i^*(b)} ({\tilde f}_i^*)^{\varphi_i^*(b)} b
$, so
\begin{align}
  \label{eq:16}
  \varphi_i(\sigma_i b) = \varphi_i^*(b) + \varphi_i\left( ({\tilde f}_i^*)^{\varphi_i^*(b)} b \right).
\end{align}
In particular, $\varphi_i(\sigma_i b) \geq   \varphi_i^*(b)$. A similar argument establishes $\varphi_i(\sigma_i b) \leq   \varphi_i^*(b)$.
\end{proof}

\begin{Lemma}{\label{lem:traptoparallel-prime}}
  We have
   $\cc^{\prec}_{d_i\delta-r_i\alpha_i}(\sigma_i(b)) = a$, $\cc^{\prec}_{\alpha_i}(\sigma_i(b)) = r_i \cdot a$, and  $\cc^{\prec}_{\beta}(\sigma_i(b)) = 0$ for all other real roots $\beta$.
\end{Lemma}
\begin{proof}
This follows immediately from Lemmas \ref{lem:trap} and \ref{lem:equality-of-phis}. 
\end{proof}

Let $\underline \mu = \cc^{\prec}_{\delta}(\sigma_i(b)) $, and let $a$ be as above, and let $b' = ({\tilde f}_i)^{r_i a} b$.

\begin{Lemma}
 \label{lem:b-prime} 
 We have  $\cc^{\prec}_{d_i\delta-r_i\alpha_i}(b') = a$,  $\cc^{\prec}_{\delta}(b') = { {\underline \mu}}$ and otherwise $\cc^{\prec}_{\beta}(b') = 0$.
\end{Lemma}
\begin{proof}
By definition, we have  
\begin{align}
  \label{eq:54}
 \sigma_i(b) = {\tilde e}_i^{r_i a} ({\tilde f}_i^*)^{r_i a} b.  
\end{align}
Therefore we also have  $b' = ({\tilde f}_i^*)^{r_i a} \left(\sigma_i(b)\right) $, which implies the lemma.
\end{proof}

\begin{Lemma} \label{lem:aa}
As above, let $\underline\mu= \cc^{\prec}_\delta(\sigma_i(b))$. Let $b'' \in B(-\infty)$ be such that $\cc_\delta^{\prec}(b'')= {\underline \mu}$ and $\cc_\beta^{\prec}(b'') = 0$ for all real roots $\beta$. By Lemma \ref{lem:trap}, we have $\cc^\preci_{\alpha_i}(b'') = r_i \cdot \cc^\preci_{d_i\delta-r_i\alpha_i}(b'') = r_i \cdot a'$ for some $a'$. 
Then $
 a' \leq a.$
Pictorially,   \\
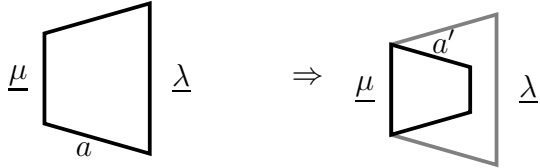
 \begin{center}
\begin{tikzpicture}[xscale=-0.7, yscale=-0.2]
\draw[line width = 0.05cm] 
(0,0)--(2,2)--(2,8)--(0,10)--cycle;

\node at (-0.6,5) {$\underline{\lambda}$};
\node at (2.5,5) {$\underline\mu$};

\node at (1.25,9.7) {$a$};
\node at (-3,5) {$\Rightarrow$};
\end{tikzpicture}
\begin{tikzpicture}[xscale=-0.7, yscale=-0.2]

\draw[line width = 0.05cm,color=gray] (0,0)--(2,2)--(2,8)--(0,10)--cycle;
\node at (-0.6,5) {$\underline\lambda$};
\node at (2.5,5) {$\underline\mu$};

\node at (1,1.8) {$a'$};
\draw[line width = 0.05cm] 
(0.5,3.5)--(2,2)--(2,8)--(0.5,6.5)--cycle;
\end{tikzpicture}
\end{center}
where the trapezoid on the left is a $2$-face of the PBW polytope of $\sigma_i(b)$, and the smaller trapezoid on the right is the corresponding $2$-face of the PBW polytope of $b''$.
\end{Lemma}

\begin{proof}
By Corollary \ref{cor:im-sym} we have $\cc^{\preci}(\sigma_i(b)) = \cc^{\prec}(b) = \underline \lambda$. Let $b' = ({\tilde f}_i)^{r_i a} b$. 
By Lemma \ref{lem:b-prime},
\begin{align}
  \label{eq:21}
 b' \equiv S^{\prec,\underline \mu} E^{\prec,(a)}_{d_i\delta - r_i\alpha_i} .
\end{align}
Choosing an approximating one-row order, we may assume that  $\prec$ is one-row. The {\pcom typo} we can write $E^{\prec}_{d_i\delta - r_i\alpha_i} = T_{w}(E_j)$ for some $w \in W$ and $j \in I$. Choose a reduced decomposition $w = s_{i_1} \cdots s_{i_\ell}$, and let $\sigma^{*}_{w^{-1}} = \sigma^*_{i_\ell} \cdots \sigma^*_{i_1}$. 
Then $b'' = ({\tilde f}_j^*)^{ a} \sigma_{w^{-1}}^* b'$ and 
\begin{align}
  \label{eq:22}
  a' = \varphi_j(({\tilde f}_j^*)^{ a} \sigma_{w^{-1}}^* b') \leq \varphi_j( \sigma_{w^{-1}}^* b') = a,
\end{align}
where the inequality is by the characterization of $B(-\infty)$ in Proposition \ref{prop:comb-char}.
\end{proof}

  \begin{Lemma}
    \label{lem:a-eq-zero}
Let $i \in \bar I$, and let $\prec$ be a convex order of the standard coarse type with $\alpha_i$ minimal. Let $\underline \lambda = (\lambda^{(1)}, \cdots, \lambda^{(n)})$ be a multipartition with $\lambda^{(i)} \neq 0$ and suppose condition (I) is known for all multipartitions of weight less than $\underline \lambda$. Then $\varphi^*_{i}\left(b(\underline\lambda,\prec)\right) > 0$.
  \end{Lemma}

  \begin{proof}
    
Suppose instead that $\varphi^*_{i}\left(b(\underline\lambda,\prec)\right)=0$. Since the weight of $b(\underline\lambda,\prec)$ is imaginary, $\sigma_i(b(\underline\lambda,\prec)) = b(\underline\lambda,\prec)$, and so
\begin{align}
  \label{eq:13}
  T_i (S^{\prec,\underline \lambda}) \equiv S^{\prec,\underline \lambda} \mod q_s^{-1}\cL.
\end{align}

Recall that
\begin{align}
  \label{eq:43}
S^{\prec,\underline \lambda} = \prod_{j  \neq i} S^{\prec,\lambda^{(j)}}_j \cdot  S^{\prec, \lambda^{(i)}}_i
\end{align}
and by Corollary \ref{cor:im-sym}, that:
\begin{align}
  \label{eq:42}
S^{\preci,\underline \lambda} = T_i (S^{\prec,\underline \lambda})
\end{align}

Applying condition (I) inductively we have
\begin{align}
  \label{eq:49}
T_i \left(\prod_{j  \neq i} S^{\prec,\lambda^{(j)}}_j \right) = \left(\prod_{j  \neq i} S^{\preci,\lambda^{(j)}}_j \right) \equiv \left(\prod_{j  \neq i} S^{\prec,\lambda^{j}}_j \right)  \mod q_s^{-1} \cL.
\end{align}
Therefore 
\begin{align}
  \label{eq:50}
 T_i \left(\prod_{j  \neq i} S^{\prec,\lambda^{j}}_j \right) =   \left(\prod_{j  \neq i} S^{\prec,\lambda^{j}}_j \right) + x_{re} + x_{im}
\end{align}
where  $x_{re}$ is a linear combination of $\prec$-PBW basis vectors with Lusztig data involving real parts,  $x_{im}$ is a linear combination of $\prec$-PBW basis vectors with purely imaginary Lusztig, and $x_{re}, x_{im} \in q^{-1} \cL$.

By \cite[Theorem 4.17]{MT}, there is a PBW basis vector $L(\dd,\prec)$ with $\dd_{\alpha_i} =  \lambda^{(i)}_1 \neq 0$ such that 
\begin{align}
  \label{eq:47}
T_i S^{\prec, \lambda^{(i)}}_i =  L(\dd,\prec) + y_{re} + y_{im}
\end{align}
where $y_{re}$ is a linear combination of $\prec$-PBW basis vectors involving Lusztig data with real parts, $y_{im}$ is a linear combination of $\prec$-PBW basis vectors. Then
\begin{align}
  \label{eq:48}
  T_i S^{\prec,\underline \lambda} \hspace{-0.1cm} =  \hspace{-0.1cm}  T_i  \hspace{-0.1cm}  \left( \prod_{j  \neq i} S^{\prec,\lambda^{j}}_j \right)   T_i S^{\prec, \lambda^{(i)}}_i  \hspace{-0.2cm}  =  \hspace{-0.cm} 
 \left( \prod_{j  \neq i} S^{\prec,\lambda^{(j)}}_j  + x_{re} + x_{im} \right) \hspace{-0.1cm}   \left( L(d,\prec) + y_{re} + y_{im} \right).
\end{align}
Rewriting in the $\prec$-PBW basis, by Corollary \ref{cor:beck-nakajima-upper-triangularity}, terms with purely imaginary Lusztig data can only arise from the product of PBW basis vectors with purely imaginary Lusztig data. 

Let $\cL_0$ denote the $\cA$-span of $\prec$-PBW basis vectors with purely imaginary Lusztig data. We know that the imaginary root vectors that form a $\cA$-basis of $\cL_0$ multiply exactly as Schur functions; in particular, $q_s^{-1}\cL_0$ is closed under multiplication by $\cL_0$. Therefore, the purely imaginary terms that appear when we expand \eqref{eq:48} in the $\prec$-PBW basis must lie in $q_s^{-1} \cL_0$. This contradicts \eqref{eq:13}.
\end{proof}

\begin{Definition} \label{order-on-mp}
Define a partial order on multipartitions by, for 
$\underline \lambda = (\lambda^{(1)}, \cdots, \lambda^{(n)})$ and $\underline {\tilde \lambda} = (\tilde\lambda^{(1)}, \cdots, {\tilde\lambda}^{(n)})$,  $\underline\lambda \geq \underline{\tilde\lambda}$ if
\begin{itemize}
\item $\wt(\underline{\tilde\lambda}) = \wt(\underline{\lambda})$, and
\item ${\tilde\lambda}^{(i)}$ dominates $\lambda^{(i)}$ for each $i \in \bar I$. 
\end{itemize}
\end{Definition}

  Let $\prec$ be a convex order of the standard coarse type with $\alpha_i$ minimal, and let $\underline\lambda  = (\lambda^{(1)}, \cdots, \lambda^{(n)})$ be a multipartition. Applying Lemma \ref{lem:traptoparallel-prime},
  \begin{align}
    \label{eq:11}
    \sigma_i b(\underline\lambda,\prec) = b(\cc,\prec)
  \end{align}
where $ \cc_\delta = {\underline\mu} = (\mu^{(1)}, \cdots, \mu^{(n)})$, $\cc_{\alpha_i} = r_i \cdot \cc_{d_i\delta - r_i\alpha_i} = r_i a$ for some non-negative integer $a$, and $\cc_{\beta}= 0$ otherwise.  Define an endomorphism $\Phi_i$ on the set multipartitions by
  \begin{align}
    \label{eq:12}
   \Phi_i : \underline \lambda  \mapsto {\underline \mu} \sqcup_i a,
  \end{align}
  where $\underline \mu \sqcup_i a$ is obtained by adding a part of size $a$ to the $i$-th partition in $\underline \mu$.

  \begin{Lemma} \label{lem:bi} $\Phi_i$ is the identity map. 
  \end{Lemma}

\begin{proof}
Since the set of multipartitions of a given weight is finite, it suffices to show that $\Phi_i$ is an injective increasing function, meaning that, for all multipartitions $\underline \lambda$,   \begin{align}
    \label{eq:013}
    \underline \lambda \leq \Phi_i(\underline \lambda). 
  \end{align}
Proceed by induction on the weight of $\underline \lambda$. It is clear that the map $\underline \lambda \mapsto (\underline \mu, a)$ is injective, so to prove that $\Phi_i$ is an injection is suffices to prove that $a \geq { \mu}^{(i)}_1$, since then we can recover $a$ as the largest part of the $i$-th partition of $\underline \mu \sqcup_i a$.  If $a > 0$, then $a \geq { \mu}^{(i)}_1$ by Lemma \ref{lem:aa} and induction. If $a = 0$, then $\lambda^{(i)} = \emptyset$ by Lemma \ref{lem:a-eq-zero} and $\sigma_i(b) = b$, so $\underline\mu = \underline\lambda$. Hence, $\mu^{(i)}_1 = 0$, and we have our inequality $a \geq { \mu}^{(i)}_1$. Furthermore inequality \eqref{eq:013} is clear when $a=0$. It remains to prove \eqref{eq:013} when $a > 0$.

Let $b = b(\underline\lambda,\prec)$, and let $b' = ({\tilde f_i}^*)^{r_i a} b = ({\tilde f_i})^{r_i a} \sigma_i(b)$. By Lemma \ref{lem:b-prime}, $b'= b(\cc', \prec)$ where $\cc'_{\delta}=\underline \mu, \cc'_{d_i\delta-r_i\alpha_i}= a$, and $\cc'_\beta = 0$ otherwise. 
By upper triangularity of the PBW basis with respect to the canonical basis,
\begin{align}
  \label{eq:14}
 b' = S^{\underline \mu} E_{d_i\delta-r_i\alpha_i}^{(a)}+ \sum_{\cc'' >_\ell \cc'} a_{\cc',\cc''}L(\cc'',\prec).  
\end{align}
The fact that the inequality ${\cc'' >_\ell \cc'}$, is strict follows because otherwise, by definition, ${\cc'' >_r \cc'}$, and the only way for this to occur is for $\cc''_{d_i \delta - r_i \alpha_i} > a$, which is impossible since $\wt(\cc'') = \wt(\underline \mu) + a \cdot (d_i \delta - r_i \alpha_i)$. 
Thus
\begin{align}
\label{eqa:1} 
b' E_i^{(r_i a)} &  =S^{\underline \mu} E_{d_i\delta-r_i\alpha_i}^{(a)} E_i^{(r_i a)} +\sum_{\cc'' >_\ell \cc'} a_{\cc',\cc''}L(\cc'',\prec) \cdot E_i^{(a)}.
\end{align}
By \cite[Proposition 3.17]{BN} (the statement there is only for the order $\prec_0$, but we can extend it to any order of the standard coarse type using Proposition \ref{prop:lop}), \begin{align}
  \label{eq:15}
E_{d_i\delta-r_i\alpha_i}^{(a)} E_i^{(r_i a)} = S^{(a)_i} + \sum_{\cc''' : \cc''' \text{ has some real part}} b_{\cc'''}L(\cc''',\prec).
\end{align}
Here $(a)_i$ is the multipartition whose $i$-th partition is the one-part partition of size $a$, and whose other partitions are all empty.
Combining and applying Lemma \ref{lem:upper-triangularity-of-multiplication}, \begin{align}
   b' E_i^{(r_i a)} &=  S^{\underline \mu} S^{(a)_i}+ \sum_{\cc'''' : \cc'''' \text{ has some real part}} d_{\cc''''}L(\cc'''',\prec) \\
\label{eqa:3} &= S^{ \underline \mu \sqcup (a)_i}+ \sum_{\underline \nu <   \underline \mu \sqcup (a)_i} a_{\underline \nu}S^{\underline \nu}+ \sum_{\cc'''' : \cc'''' \text{ has some real part}} d_{\cc''''}L(\cc'''',\prec).
 \end{align}
 Here, \eqref{eqa:3} is the Pieri rule, and the coefficients $a_{\underline \nu}$ are zero or one accordingly. 
Using the triangularity with the canonical basis, 
 $$ b' E_i^{(r_ia)}= b( \underline \mu \sqcup (a)_i,\prec) + \sum_{\underline \nu <   \underline \mu \sqcup (a)_i} a_\nu b( \underline \nu, \prec) + \text{ canonical basis terms with real parts}.$$
By Proposition \ref{CanonicalBasisCrystalFormula-prime} the correct canonical basis element must show up here. By construction this element is $b(\prec, \underline \lambda)$, so we have shown that $\underline \lambda \leq \underline \mu \sqcup (a)_i$. 
\end{proof}

By Corollary~\ref{cor:im-sym} and Proposition \ref{prop:Saito}, we have that $b(\underline\lambda,\prec^{s_i}) = \sigma_i b(\lambda,\prec)$. By Lemma \ref{lem:bi}, we conclude that  $b(\underline\lambda,\prec^{s_i}) = b( \cc,\prec)$ where 
$  \cc_\delta$ is the multipartition obtained from $\underline\lambda$ by removing the largest part from $\lambda^{(i)}$, $\cc_{\alpha_i} = r_i \cdot \cc_{d_i\delta - r_i\alpha_i} = r_i \cdot \lambda^{(i)}_1$, and $ \cc_{\beta} = 0$ for all other real roots. This exactly establishes Condition (I) of Theorem \ref{th:MV-def}, completing the proof of Theorem \ref{thm:PBWMV}.

\begin{Remark}
The above gives an independent construction of a map satisfying Theorem \ref{th:MV-def}, and can therefore be used to develop affine MV polytopes independently of the preprojective algebra construction of \cite{BKT} or the KLR algebra construction of \cite{TW}. 
\end{Remark}


\begin{thebibliography}{[HKKPTVVZ]}


\bibitem[Aka]{Aka}
T. Akasaka,
``{An integral {PBW} basis of the quantum affine algebra of type {$A^{(2)}_2$}},''
{\em Publ. Res. Inst. Math. Sci.},
{\bf 38},
{2002}.
\newblock \arxiv{math/0105170v1}

\bibitem[And]{And} 
J. Anderson,
``{A polytope calculus for semisimple groups},''
{\em Duke Math. J.},
{\bf 116},
(2003), 567--588.
\newblock \arxiv{math/0110225} 

\bibitem[BDKT]{BDKT:13} 
P. Baumann, T. Dunlap, J. Kamnitzer, and P. Tingley,
``Rank 2 affine MV polytopes,'' {\it Represent. Theory } {\bf17} (2013), 442--468
\arxiv{1202.6416}

\bibitem[BG]{baumanngaussent}
P. Baumann and S.~Gaussent, ``{On {M}irkovi\'c-{V}ilonen cycles and crystal combinatorics},'' {\em Represent. Theory}, {\bf 12} (2008), 83--130.
\arxiv{math/0606711}

\bibitem[BK]{BK}
P. Baumann and J. Kamnitzer,
``Preprojective algebras and MV polytopes,'' {\it Represent. Theory }{\bf16} (2012), 152--188.
\arxiv{1009.2469}

\bibitem[BKT]{BKT} 
P. Baumann, J. Kamnitzer, and P. Tingley, ``Affine Mirkovi\'c-Vilonen polytopes,'' {\em Publ. Math. Inst. Hautes \'Etudes Sci.} {\bf 120} (2014), 113--205. \arxiv{1110.3661}.
 
  \bibitem[Beck1]{Beck1}
J. Beck, 
``{Braid group action and quantum affine algebras},''
{\em Comm. Math. Phys.},
{\bf 165}, {(1994)}.
\arxiv{hep-th/9404165}.

  
  \bibitem[Beck2]{Beck2}
J. Beck, 
``{Convex Bases of PBW Type for Quantum Affine Algebras},''
{\em Comm. Math. Phys.},
{\bf 165}, {(1994)}.
\arxiv{hep-th/9407003}.

\bibitem[BCP]{BCP}
J. Beck, V. Chari and A. Pressley,
``{An algebraic characterization of the affine canonical basis},''
{\em Duke Math. J.},
{\bf 99},
(1999),
455--487.
\arxiv{math/9808060}


\bibitem[BN]{BN}
J. Beck and H. Nakajima,
``{Crystal bases and two-sided cells of quantum affine algebras},''
{\em Duke Math. J.},
{\bf 123},
(2004),
335--402.
\arxiv{math/0212253}

\bibitem[BZ]{BZ}
A. Berenstein and A. Zelevinsky, 
``{Tensor product multiplicities, canonical bases, and totally positive varieties},''
{\em Invent. Math.} {\bf 143} (2001),  77--128.
\arxiv{math/9912012}

\bibitem[Dam]{Dam}
I.~Damiani, 
``A basis of type Poincar\'e-Birkhoff-Witt for the quantum algebra of $\widehat{sl}(2)$,'' {\it J. Algebra} {\bf 161} (1993), no. 2, 291--310. 

\bibitem[Ito]{Ito} K. Ito, ``The classification of convex orders on affine root systems,'' {\em Communications in Algebra}. Vol 29, Issue 12, 2001.
\arxiv{math/9912020}
  
\bibitem[Kac]{Kac}
V. Kac,
{\em Infinite-dimensional {L}ie algebras},
{Cambridge University Press, Cambridge},
{1990}.

\bibitem[Kam07]{Kam2}
J. Kamnitzer,
``{The crystal structure on the set of {M}irkovi\'c-{V}ilonen polytopes},''
{\em Adv. Math.} {\bf 215} (2007), 66--93.
\arxiv{math/0505398}


\bibitem[Kam10]{Kam1}
J. Kamnitzer,
``{Mirkovi\'{c}-Vilonen cycles and polytopes},''
{\em Ann. of Math. (2)} {\bf 171} (2010), 731--777.
\arxiv{math/0501365}

\bibitem[Kas]{Kashiwara:1995}
M.~Kashiwara,
\newblock ``On crystal bases,''
\newblock {\em CMS Conf. Proc.}, 16, 1995.

\bibitem[KS]{KS:1997}
M.~Kashiwara and Y.~Saito,
\newblock ``Geometric construction of crystal bases,''
\newblock {\em Duke Math. J.}, 89(1):9--36, 1997.
\newblock \arxiv{q-alg/9606009}.

 \bibitem[Lec04]{Lec} B. Leclerc, ``Dual canonical bases, quantum shuffles and q-characters,'' 
 {\it Math. Z.} {\bf 246} (2004), Issue 4, pp 691--732. \arxiv{math/0209133v3}

\bibitem[Lus90]{LusCanonical}
G. Lusztig, ``{Canonical bases arising from quantized enveloping algebras},'' {\em J. Amer. Math. Soc.} {\bf 3} (1990)  447--498. 

\bibitem[Lus91]{Lus91}
G. Lusztig, ``{Quivers, perverse sheaves, and quantized enveloping
  algebras},'' {\em J. Amer. Math. Soc.} \textbf{4} (1991), no.~2, 365--421.
  \MR{MR1088333 (91m:17018)}

\bibitem[Lus92]{Lus}
G. Lusztig. {\em Introduction to Quantum Groups}, \textit{Progr. in Math.} \textbf{105}, Birkh\"auser, 1992.

\bibitem[Lus96]{Lus96}
G. Lusztig, ``Braid group action and canonical bases,'' {\it Adv. Math.} {\bf 122} (1996), 237--261.


\bibitem[McN]{McN}
P. McNamara, 
``{Representations of Khovanov-Lauda-Rouquier algebras III: Symmetric Affine Type.}''
\arxiv{1407.7304}

\bibitem[MV]{MirkovicVilonen04}
I.~Mirkovi\'c, K.~Vilonen, ``{Geometric Langlands duality
and representations of algebraic groups over commutative rings,}''
{\em Ann.\ of Math. (2)} \textbf{166} (2007), 95--143.

\bibitem[MT]{MT}
D. Muthiah and P. Tingley, ``Affine PBW bases and MV polytopes in rank 2,''    {\em Selecta Math. (N.S.)}, {\bf 20}, {(2014)}, {no. 1}, {237--260}.
\arxiv{1209.2205}

\bibitem[Sai]{Sai}
Y. Saito, 
``{PBW basis of quantized universal enveloping algebras,}''
\textit{Publ. Res. Inst. Math. Sci.} {\bf 30} (1994), no. 2, 209--232. 

\bibitem[Tin]{LusNotes}
P. Tingley, 
``Elementary construction of Lusztig's canonical basis.'' \textit{AMS Contemporary Math
volume on Groups, Rings, Group Rings and Hopf Algebras, dedicated to Don Passman}, To appear.
\arxiv{1602.04895}

\bibitem[TW]{TW}
P. Tingley and B. Webster,
``{Mirkovi\'c-Vilonen polytopes and Khovanov-Lauda-Rouquier algebras},'' 
{\em Compos. Math.}, {\bf 152}, {(2016)}, {no. 8}, {1648--1696}.
\arxiv{1210.6921}

\end{thebibliography}
\end{document}